\documentclass[a4paper, 10pt, parskip=half]{scrartcl}

\usepackage[utf8]{inputenc}
\usepackage[T1]{fontenc}
\usepackage{lmodern}
\usepackage{amsmath}
\usepackage{amssymb}
\usepackage{amsthm}
\usepackage{amsfonts}
\usepackage{dsfont}
\usepackage{mathtools}
\usepackage{bbm}
\usepackage{marginnote}
\usepackage{hyperref}
\hypersetup{
  colorlinks=false,
  pdfborder={0 0 0},
  pdftitle={Blow-up profiles in quasilinear fully parabolic Keller--Segel systems},
  pdfauthor={Mario Fuest},
  pdfkeywords={blow-up profile; nonlinear diffusion; chemotaxis},
  bookmarksopen=true,
}

\RequirePackage{geometry}
\newcommand{\R}{\mathbb{R}}

\newcommand{\N}{\mathbb{N}}

\newcommand{\ur}[1]{\mathrm{#1}}
\newcommand{\ure}{\ur e}

\newcommand{\eps}{\varepsilon}

\newcommand{\gt}{>}
\newcommand{\lt}{<}

\newcommand{\defs}{\coloneqq}
\newcommand{\sfed}{\eqqcolon}

\newcommand{\ra}{\rightarrow}

\newcommand{\nea}{\nearrow}

\newcommand{\sea}{\searrow}

\newcommand{\ol}{\overline}
\newcommand{\ul}{\underline}

\newcommand{\wh}{\widehat}

\newcommand{\dx}{\,\mathrm{d}x}

\newcommand{\ddt}{\frac{\mathrm{d}}{\mathrm{d}t}}

\newcommand{\embed}{\hookrightarrow}

\newcommand{\hp}{\hphantom}
\newcommand{\pe}{\mathrel{\hp{=}}}

\newcommand{\tmax}{T_{\max}}
\newcommand{\tmaxe}{T_{\max, \eps}}

\newcommand{\intom}{\int_\Omega}

\newcommand{\ombar}{\ol \Omega}

\newcommand{\leb}[1]{{L^{#1}(\Omega)}}
\newcommand{\sob}[2]{{W^{#1, #2}(\Omega)}}
\newcommand{\con}[1]{{C^{#1}(\ombar)}}

\newcommand{\pu}{\mathbbmss p}
\newcommand{\xw}{\left( |x|^{-\frac{(m-1)\alpha}{2}} w^{\frac{p+m-1}{2}} \right)}

\newcommand{\ue}{u_\eps}
\newcommand{\uej}{u_{\eps_j}}
\newcommand{\uet}{u_{\eps t}}
\newcommand{\ve}{v_\eps}
\newcommand{\vej}{v_{\eps_j}}
\newcommand{\vet}{v_{\eps t}}
\newcommand{\Ge}{G_\eps}
\newcommand{\we}{w_\eps}
\newcommand{\ze}{z_\eps}

\renewcommand{\paragraph}[2][.]{\textbf {#2#1}}

\newtheoremstyle{nplain}
  {\topsep}   
  {\topsep}   
  {\itshape}  
  {0pt}       
  {\bfseries} 
  {.}         
  {5pt plus 1pt minus 1pt} 
  {\thmnumber{#2 }\thmname{#1}\thmnote{ (#3)}} 

\newtheoremstyle{ndefinition}
  {\topsep}   
  {\topsep}   
  {}   	      
  {0pt}       
  {\bfseries} 
  {.}         
  {5pt plus 1pt minus 1pt} 
  {\thmnumber{#2 }\thmname{#1}\thmnote{ (#3)}} 

\makeatletter
\renewenvironment{proof}[1][\proofname]{\par
  \pushQED{\qed}%
  \normalfont \topsep0\p@\relax
  \trivlist
  \item[\hskip\labelsep\scshape
  #1\@addpunct{.}]\ignorespaces
}{%
  \popQED\endtrivlist\@endpefalse
}
\makeatother

\newtheorem{base}{Base}[section]
\numberwithin{equation}{section}

\theoremstyle{nplain}
\newtheorem{theorem}[base]{Theorem} \newtheorem*{theorem*}{Theroem}
\newtheorem{lemma}[base]{Lemma} \newtheorem*{lemma*}{Lemma}
 \newtheorem*{prop*}{Proposition}
 \newtheorem*{cor*}{Corollary}

\theoremstyle{ndefinition}
 \newtheorem*{definition*}{Definition}
 \newtheorem*{example*}{Example}
 \newtheorem*{cond*}{Condition}

\newtheorem{remark}[base]{Remark} \newtheorem*{remark*}{Remark}

\begin{document}
\setkomafont{title}{\normalfont \Large}
\title{Blow-up profiles in quasilinear fully parabolic Keller--Segel systems}
\author{
Mario Fuest\footnote{fuestm@math.uni-paderborn.de}\\
{\small Institut f\"ur Mathematik, Universit\"at Paderborn,}\\
{\small 33098 Paderborn, Germany}
}
\date{}

\maketitle

\KOMAoptions{abstract=true}
\begin{abstract}
  \noindent
  We examine finite-time blow-up solutions $(u, v)$ to 
  \begin{align} \label{prob:star} \tag{$\star$}
    \begin{cases}
      u_t = \nabla \cdot (D(u, v) \nabla u - S(u, v) \nabla v), \\
      v_t = \Delta v - v + u
    \end{cases}
  \end{align}
  in a ball $\Omega \subset \R^n$, $n \ge 2$,
  where $D$ and $S$ generalize the functions
  \begin{align*}
    D(u, v) = (u+1)^{m-1}
    \quad \text{and} \quad
    S(u, v) = u (u+1)^{q-1}
  \end{align*}
  with $m, q \in \R$.
  We show that if $m \gt \frac{n-2}{n}$ as well as $m-q \gt -\frac1n$
  and $(u, v)$ is a nonnegative, radially symmetric classical solution
  to \eqref{prob:star}
  blowing up at $T_{\textrm{max}} \lt \infty$,
  then there exists a so-called \emph{blow-up profile}
  $U \colon \Omega \setminus \{0\} \ra [0, \infty)$ satisfying
  \begin{align*}
   u(\cdot, t) \ra U \quad \text{in $C_{\textrm{loc}}^2(\ol \Omega \setminus \{0\})$ as $t \nea T_{\textrm{max}}$}.
  \end{align*}
  Moreover, for all $\alpha \gt n$ with
  \begin{align*}
    \alpha \gt \frac{n(n-1)}{(m-q)n + 1}
  \end{align*}
  we can find $C \gt 0$ such that
  \begin{align*}
    U(x) \le C |x|^{-\alpha}
  \end{align*}
  for all $x \in \Omega$.\\[0.5pt]
  \textbf{Key words:} {blow-up profile; nonlinear diffusion; chemotaxis}\\
  \textbf{AMS Classification (2010):} {35B40 (primary); 35K40, 35K65, 92C17 (secondary)}
\end{abstract}

\section{Introduction}
The possibility of (finite-time) blow-up constitutes one of the most striking features of the quasilinear system 
\begin{align} \label{prob:ks} \tag{KS}
  \begin{cases}
    u_t = \nabla \cdot (D(u, v) \nabla u - S(u, v) \nabla v),              & \text{in $\Omega \times (0, T)$}, \\
    v_t = \Delta v - v + u,                                                & \text{in $\Omega \times (0, T)$}, \\
    (D(u, v) \nabla u - S(u, v) \nabla v) \cdot \nu = \partial_\nu v  = 0, & \text{on $\partial \Omega \times (0, T)$}, \\
    u(\cdot, 0) = u_0, v(\cdot, 0) = v_0,                                  & \text{in $\Omega$},
  \end{cases}
\end{align}
proposed by Keller and Segel \cite{KellerSegelTravelingBandsChemotactic1971}
to model chemotaxis, that is, the directed movement of bacteria or cells towards a chemical signal,
and attracting interest of mathematicians for nearly half a century
(see for instance \cite{BellomoEtAlMathematicalTheoryKeller2015} for a recent survey).

Therein $\Omega \subset \R^n$, $n \in \N$, is a smooth, bounded domain,
$T \in (0, \infty]$
and $u_0, v_0 \colon \ombar \ra [0, \infty)$
as well as $D, S \colon [0, \infty]^2 \ra [0, \infty)$  are sufficiently smooth given functions,
the most classical choices being $D \equiv 1$ and $S(u, v) = u$.

For these selections, namely, solutions blowing up in finite time have been constructed in
two- \cite{HerreroVelazquezBlowupMechanismChemotaxis1997} and
higher- \cite{WinklerFinitetimeBlowupHigherdimensional2013} dimensional balls.
On the other hand, if $n = 1$ \cite{OsakiYagiFiniteDimensionalAttractor2001},
if $n = 2$ and $\intom u_0 \lt 4\pi$ (or $\intom u_0 \lt 8\pi$ in the radially symmetric setting)
\cite{NagaiEtAlApplicationTrudingerMoserInequality1997}
or if $n \ge 3$ and $\|u_0\|_{\leb{\frac n2}} + \|v_0\|_{\sob1n}$ is sufficiently small \cite{CaoGlobalBoundedSolutions2014},
all solutions are global in time and remain bounded.
We should also note that if one replaces the second equation in \eqref{prob:ks} by a suitable elliptic counterpart,
finite-time blow-up results have been achieved already in the 1990s
\cite{HerreroEtAlFinitetimeAggregationSingle1997, JagerLuckhausExplosionsSolutionsSystem1992, NagaiBlowupRadiallySymmetric1995}.

Motivated inter alia by the desire to model volume-filling effects,
it has been suggested to consider certain nonlinear functions $D \equiv D(u)$ and $S \equiv S(u)$ instead
\cite{HillenPainterUserGuidePDE2009, PainterHillenVolumefillingQuorumsensingModels2002, WrzosekVolumeFillingEffect2010}
and, in order to account for
immotility in absence of bacteria \cite{FuEtAlStripeFormationBacterial2012, LeyvaEtAlEffectsNutrientChemotaxis2013}
or receptor-binding and saturation effects \cite{KalininEtAlLogarithmicSensingEscherichia2009, HillenPainterUserGuidePDE2009},
one might also (need to) choose functions $D$ and $S$ explicitly depending on $v$.

For the sake of exposition, we will for now confine ourselves with the prototypical choices
$D(u, v) = (u + 1)^{m-1}$ and $S(u, v) = u (u + 1)^{q-1}$ for certain $m, q \in \R$,
but remark that all the works cited below allow for more general choices of $D$ and $S$ as well.

Regarding the question of global-in-time boundedness, the number $\frac{n-2}{n}$ is critical:
If $\Omega \subset \R^n$, $n \in \N$, is a smooth, bounded domain and $m-q \gt \frac{n-2}{n}$,
then all solutions to \eqref{prob:ks} are global in time and bounded
\cite{HorstmannWinklerBoundednessVsBlowup2005, IshidaEtAlBoundednessQuasilinearKeller2014, TaoWinklerBoundednessQuasilinearParabolic2012}.
Conversely, if $\Omega \subset \R^n$, $n \ge 2$, is a ball and $m-q \lt \frac{n-2}{n}$,
there exist initial data such that the corresponding solution blows up in either finite or infinite time
\cite{HorstmannWinklerBoundednessVsBlowup2005, WinklerDoesVolumefillingEffect2009}.

If in addition to $m-q \lt \frac{n-2}{n}$ one assumes $n \ge 3$
as well as either $m \ge 1$ (and hence $q \gt \frac2n \gt 0$) or $m \in \R$ and $q \ge 1$,
finite-time blow-up is possible
\cite{CieslakStinnerFinitetimeBlowupGlobalintime2012, CieslakStinnerFiniteTimeBlowupSupercritical2014, CieslakStinnerNewCriticalExponents2015},
while for $q \le 0$ solutions are always global in time \cite{WinklerGlobalClassicalSolvability2019}.
Whether solutions may blow up in finite time given $m - q \lt \frac{n-2}{2}$ and $q \gt 0$ but $q \lt 1$ or $m \lt 1$
is, to the best of our knowledge, still an open question.

The picture is more complete
if one replaces the second equation in \eqref{prob:ks} with a suitable elliptic equation.
Again solutions are global and bounded provided that $m - q \gt \frac{n-2}{2}$
and in the radial symmetric setting there exist unbounded solutions if $m - q \lt \frac{n-2}{2}$.
Additionally, it is known for which parameters finite-time blow-up may occur:
If $q \le 0$, these solutions are always global, while for $q \gt 0$ finite-time blow-up is possible
\cite{LankeitInfiniteTimeBlowup2017, WinklerDjieBoundednessFinitetimeCollapse2010}.
An obvious conjecture, stated for instance in \cite{WinklerGlobalClassicalSolvability2019},
is that the same holds true for the fully parabolic system \eqref{prob:ks}.

A natural next step is to examine the qualitative behavior of (finite- or infinite-time) blow-up solutions in more detail.
While far from exhaustive, some results in this regard have been obtained for the classical Keller--Segel system, that is,
for $D \equiv 1$ and $S(u, v) = u$.

In the two-dimensional settings some blow-up solutions collapse to a Dirac-type singularity
(see \cite{HerreroVelazquezBlowupMechanismChemotaxis1997, NagaiEtAlChemotacticCollapseParabolic2000}
or also \cite{SenbaSuzukiChemotacticCollapseParabolicelliptic2001} for similar results for the parabolic--elliptic case).
Additionally, for all $n \ge 2$, temporal blow-up rates (even for $S(u, v) = u^q, q \in (0, 2)$) have been established
\cite{MizoguchiSoupletNondegeneracyBlowupPoints2014}
and it is known that $\{u^\frac{n}{2}(\cdot, t) \colon t \in (0, \tmax)\}$ cannot be equi-integrable,
where $\tmax$ denotes the blow-up time \cite{CaoInterpolationInequalityIts2017}.

Quite recently, the questions whether spatial blow-up profiles exist,
that is, whether $U \defs \lim_{t \nea \tmax} u(\cdot, t)$, $\tmax$ again denoting the blow-up time, is meaningful in some sense,
and, if this is indeed the case, properties of $U$
have been studied.

Choosing $\Omega$ to be a ball in two or more dimensions, $D \equiv 1$ and $S(u, v) = u$,
it has been shown in \cite{WinklerBlowupProfilesLife}
that for all nonnegative, radially symmetric solutions blowing up at $\tmax \lt \infty$
there exists a blow-up profile $U$ in the sense that
$u(\cdot, t) \ra U$ in $C_{\mathrm{loc}}^2(\ombar \setminus \{0\})$ as $t \nea \tmax$.
Moreover, an upper estimate is available for $U$: For any $\eta \gt 0$ one can find $C \gt 0$ with
\begin{align*}
  U(x) \le C |x|^{-n(n-1) - \eta} \quad \text{for all $x \in \Omega$}.
\end{align*}

If one simplifies \eqref{prob:ks} by not only setting $D \equiv 1$ and $S(u, v) = u$
but also replacing the second equation therein with $0 = \Delta v - \frac{1}{|\Omega|} \intom u_0 + u$,
more detailed information is available.
In \cite{SoupletWinklerBlowupProfilesParabolicelliptic2018},
the authors consider $\Omega \defs B_R(0) \subset \R^n$, $R \gt 0, n \ge 3$,
and construct a large class of initial data
for which the corresponding solutions $(u, v)$ blow up in finite time.
The blow-up profile $U \defs \lim_{t \nea \tmax} u(\cdot, t)$ exists pointwise and
\begin{align*}
  U(x) \le C |x|^{-2} \quad \text{for all $x \in \Omega$}
\end{align*}
holds for some $C \gt 0$, wherein the exponent $2$ is optimal.
Furthermore, the same paper also provides certain lower bounds for $U$.

Up to now, however, in the case of nonlinear diffusion there seems to be nearly no information available
regarding behavior of finite-time blow-up solutions to \eqref{prob:ks} at their blow-up time.
The present paper aims to be a first step towards closing this gap.

\paragraph{Main results}
At first, we will deal with (a slight generalization of) the first sub-problem in \eqref{prob:ks}
and derive pointwise estimates for its solutions.

\begin{theorem} \label{th:pw_scalar}
  Let $\Omega \subset \R^n$, $n \ge 2$, be a smooth, bounded domain with $0 \in \Omega$
  as well as
  \begin{align} \label{eq:pw_scalar:params}
    m, q \in \R,
    K_{D,1}, K_{D,2}, K_S, K_f, M, L, \beta \gt 0,
    \theta \gt n,
    \pu \ge 1
  \end{align}
  be such that
  \begin{align} \label{eq:pw_scalar:cond_m_q}
    m-q \in \left(\frac{\pu}\theta -\frac{\pu}n, \frac{\pu}{\theta} + \frac{\beta \pu-\pu}{n} \right]
    \quad \text{and} \quad
    m \gt \frac{n-2\pu}{n}.
  \end{align}

  Then for any
  \begin{align} \label{eq:pw_scalar:cond_alpha}
    \alpha \gt\frac{\beta}{m-q + \frac{\pu}{n} - \frac{\pu}{\theta}}
  \end{align}
  we can find $C \gt 0$ with the following property:

  Suppose that for some $T \in (0, \infty]$ the function $u \in C^0(\ombar \times [0, T)) \cap C^{2, 1}(\ombar \times (0, T))$
  is nonnegative, fulfills
  \begin{align} \label{eq:pw_scalar:p}
    \sup_{t \in (0, T)} \intom u^{\pu} \le M
  \end{align}
  and is a classical solution of
  \begin{align} \label{prob:scalar_eq}
    \begin{cases}
      u_t \le \nabla \cdot (D(x, t, u) \nabla u + S(x, t, u) f(x, t)), & \text{in $\Omega \times (0, T)$}, \\
      (D(x, t, u) \nabla u + S(x, t, u) f) \cdot \nu \le 0,            & \text{on $\partial \Omega \times (0, T)$}, \\
      u(\cdot, 0) \le u_0,                                             & \text{in $\Omega$},
    \end{cases}
  \end{align}
  where
  \begin{align} \label{eq:pw_scalar:reg_dsf}
    D, S \in C^1(\ombar \times (0, T) \times [0, \infty)), \quad
    f \in C^1(\ombar \times (0, T); \R^n)
    \quad \text{and} \quad
    u_0 \in \con0
  \end{align}
  satisfy (with $Q_T \defs \Omega \times (0, T)$)
  \begin{align} 
    \label{eq:pw_scalar:D1}
      &\inf_{(x, t) \in Q_T} D(x, t, \rho) \ge K_{D,1} \rho^{m-1}, \\
    \label{eq:pw_scalar:D2}
      &\sup_{(x, t) \in Q_T} D(x, t, \rho) \le K_{D,2} \max\{\rho, 1\}^{m-1}, \\
    \label{eq:pw_scalar:S}
      &\sup_{(x, t) \in Q_T} |S(x, t, \rho)| \le K_S \max\{\rho, 1\}^q
  \end{align}
  for all $\rho \gt 0$
  and
  \begin{align} \label{eq:pw_scalar:f}
   \sup_{t \in (0, T)} \intom |x|^{\theta \beta} |f(x, t)|^\theta \dx \le K_f
  \end{align}
  as well as
  \begin{align} \label{eq:pw_scalar:u0}
    u_0(x) \le L |x|^{-\alpha} \text{ for all $x \in \Omega$.}
  \end{align}
  Then
  \begin{align} \label{eq:pw_scalar:u_est}
    u(x, t) \le C |x|^{-\alpha}
    \quad \text{for all $x \in \Omega$ and $t \in (0, T)$}.
  \end{align}
\end{theorem}

\begin{remark} \label{rm:pu_1}
  For $\pu = 1$ the condition \eqref{eq:pw_scalar:p} in Theorem~\ref{th:pw_scalar} can be replaced by
  \begin{align*}
    \intom u_0 \le M
  \end{align*}
  as integrating the PDI in \eqref{prob:scalar_eq} over $\Omega$
  and integrating by parts (all boundary terms are nonpositive because of the second condition in \eqref{prob:scalar_eq})
  assert $\intom u(\cdot, t) \le \intom u_0$ for all $t \in (0, \tmax)$.
\end{remark}

As a second step, we then apply this result to radially symmetric solutions to \eqref{prob:ks} and obtain
\begin{theorem} \label{th:pw_ks}
  Let $n \ge 2$, $R \gt 0$ and $\Omega \defs B_R(0)$
  as well as
  \begin{align} \label{eq:pw_ks:params}
    m, q \in \R,
    K_{D,1}, K_{D,2}, K_S, \gt 0, M, L \gt 0
  \end{align}
  such that
  \begin{align} \label{eq:pw_ks:cond_m_q}
    m-q \in \left(-\frac1n, \frac{n-2}{n} \right]
    \quad \text{and} \quad
    m \gt \frac{n-2}{n}.
  \end{align}

  For any
  \begin{align} \label{eq:pw_ks:cond_alpha}
    \alpha \gt \ul \alpha \defs \frac{n (n-1)}{(m-q)n + 1}
  \end{align}
  and any $\beta \gt n-1$,
  there exists $C \gt 0$ with the following property:
  Let $T \in (0, \infty]$.
  Any nonnegative and radially symmetric classical solution
  $(u, v) \in \left( C^0(\ombar \times [0, T)) \cap C^{2, 1}(\ombar \times (0, T))\right)^2$ of \eqref{prob:ks}
  fulfills
  \eqref{eq:pw_scalar:u_est} and $|\nabla v(x, t)| \le C |x|^{-\beta}$
  for all $x \in \Omega$ and $t \in (0, T)$,
  provided
  \begin{align} \label{eq:pw_ks:reg}
    D, S \in C^1([0, \infty)^2), \quad
    u_0 \in \con 0
    \quad \text{and} \quad
    v_0 \in \sob1\infty
  \end{align}
  satisfy
  \begin{align} 
    \label{eq:pw_ks:D1}
      &\inf_{\sigma \ge 0} D(\rho, \sigma) \ge K_{D,1} \rho^{m-1}, \\
    \label{eq:pw_ks:D2}
      &\sup_{\sigma \ge 0} D(\rho, \sigma) \le K_{D,2} \max\{\rho, 1\}^{m-1} \quad \text{and} \\
    \label{eq:pw_ks:S}
      &\sup_{\sigma \ge 0} |S(\rho, \sigma)| \le K_S \max\{\rho, 1\}^q
  \end{align}
  for all $\rho \ge 0$
  as well as \eqref{eq:pw_scalar:u0},
  \begin{align} \label{eq:pw_ks:v_0}
    \intom u_0 \le M
    \quad \text{and} \quad
    \|v_0\|_{\sob1\infty} \le L.
  \end{align}
\end{theorem}

\begin{remark}
  \begin{enumerate}
    \item[(i)]
      Let us briefly discuss the conditions in \eqref{eq:pw_ks:cond_m_q}.
      On the one hand, observe that $m-q \sea -\frac1n$ implies $\ul \alpha \nea \infty$.
      On the other hand, \cite{TaoWinklerBoundednessQuasilinearParabolic2012} proves that all solutions to \eqref{prob:ks}
      for a large class of functions $D, S$ are global in time and bounded,
      provided $m, q \in \R$ satisfy $m-q \gt \frac{n-2}{n}$.
      In both cases a statement of the form \eqref{eq:pw_scalar:u_est} would not be very interesting.
      (However, for $m-q \gt \frac{n-2}{n}$ the statement still holds if one sets $\ul \alpha \defs n$
      because if \eqref{eq:pw_scalar:S} is fulfilled for some $q \in \R$ then also for all larger $q$.)

      The second condition in \eqref{eq:pw_ks:cond_m_q}, however, is purely needed for technical reasons
      and we conjecture that Theorem~\ref{th:pw_ks} holds even without this restriction,
      albeit the constant $C$ may then depend on $T$ as well.
      
    \item[(ii)]
      In \cite[Corollary~2.3]{FreitagBlowupProfilesRefined2018},
      it has been shown that \eqref{eq:pw_scalar:u_est}
      cannot hold for any
      \begin{align*}
        \alpha \lt \ol \alpha \defs \min\left\{
          \frac2{(1+q-m)_+},
          \frac1{(q-m)_+}
        \right\}.
      \end{align*}
      As $m-q \lt \frac{n-2}{n}$ implies $\ul \alpha \gt n \gt \ol \alpha$,
      we do not know whether \eqref{eq:pw_ks:cond_alpha} is in general optimal.
      However, in the case of $m-q = \frac{n-2}{n}$ (and $m \gt \frac{n-2}{n}$) we have $\ul \alpha = n = \ol \alpha$,
      hence at least in this extremal case the condition $\alpha \gt \ol \alpha$ is, up to equality, optimal.
  \end{enumerate}
\end{remark}

The third and final step will then consist of proving that $\lim_{t \nea \tmax} u(\cdot, t)$ and $\lim_{t \nea \tmax} v(\cdot, t)$
exist in an appropriate sense provided the diffusion mechanism in the first equation in \eqref{prob:ks} is nondegenerate.
\begin{theorem} \label{th:profile_ks}
  Let $n \ge 2$, $R \gt 0$, $\Omega \defs B_R(0)$
  and suppose that the parameters in \eqref{eq:pw_ks:params} and the functions in \eqref{eq:pw_ks:reg}
  comply with \eqref{eq:pw_scalar:u0}, \eqref{eq:pw_ks:cond_m_q} and \eqref{eq:pw_ks:D1} -- \eqref{eq:pw_ks:v_0}.
  Furthermore, suppose also that there is $\eta \gt 0$ with
  \begin{align} \label{eq:profile_ks:non_degen}
    D \ge \eta
    \quad \text{in $[0, \infty)^2$}.
  \end{align}

  Then for any nonnegative and radially symmetric classical solution
  $(u, v)$ blowing up in finite time
  in the sense that there is $\tmax \lt \infty$ such that
  \begin{align*}
    \limsup_{t \nea \tmax} \|u(\cdot, t)\|_{\leb \infty} = \infty,
  \end{align*}
  there exist nonnegative, radially symmetric $U, V \in C^2(\Omega \setminus \{0\})$ such that
  \begin{align} \label{eq:profile_ks:u_v_conv}
    u(\cdot, t) \ra U
    \quad \text{and} \quad
    v(\cdot, t) \ra V
    \qquad \text{in $C_{\mathrm{loc}}^2(\ol \Omega \setminus \{0\})$ as $t \nea \tmax$}.
  \end{align}
  Moreover, for any $\alpha \gt \ul \alpha$ (with $\ul \alpha$ as in \eqref{eq:pw_ks:cond_alpha}) and any $\beta \gt n-1$
  we can find $C \gt 0$ with the property that
  \begin{align} \label{eq:profile_ks:profile}
    U(x) \le C |x|^{-\alpha}
    \quad \text{and} \quad
    |\nabla V(x)| \le C |x|^{-\beta}
    \qquad \text{for all $x \in \Omega$}.
  \end{align}
\end{theorem}

\begin{remark}
  Obviously, Theorem~\ref{th:profile_ks} is only of interest if, given $S$ and $D$,
  there are indeed initial data leading to finite-time blow-up.
  Therefore, we stress that, for instance,
  the choices $D(\rho, \sigma) \defs (\rho + 1)^{m-1}$ and $S(\rho, \sigma) \defs \rho (\rho + 1)^{q-1}$ for $\rho, \sigma \ge 0$
  and $m \in \R, q \ge 0$ satisfying \eqref{eq:pw_ks:cond_m_q} as well as $q \ge 1$ or $m \ge 1$
  not only comply with \eqref{eq:pw_ks:reg} -- \eqref{eq:pw_ks:S} and \eqref{eq:profile_ks:non_degen} for certain parameters
  but also allow for finite-time blow-up
  \cite{CieslakStinnerFinitetimeBlowupGlobalintime2012, CieslakStinnerNewCriticalExponents2015}. 
  That is, there exist initial data $(u_0, v_0) \in \con0 \times \sob1\infty$ such that the corresponding solution to \eqref{prob:ks}
  blows up in finite time.
  As \eqref{eq:pw_scalar:u0} and \eqref{eq:pw_ks:v_0} are then obviously fulfilled for certain $L, M \gt 0$,
  we may indeed apply Theorem~\ref{th:profile_ks}.
\end{remark}

\begin{remark}
  As a final remark, let us point out that Theorem~\ref{th:profile_ks}
  includes the result in \cite[Corollary~1.4]{WinklerBlowupProfilesLife},
  as in the case of $m=1$ and $q=1$ we have $\ul \alpha = n(n-1)$.
\end{remark}

\paragraph{Plan of the paper}
The reasoning from \cite{WinklerBlowupProfilesLife},
where estimates on blow-up profiles to solutions to \eqref{prob:ks} with $D \equiv 1$ and $S(u, v) = u$ have been derived,
is to consider $w \defs \zeta^\alpha u$ with $\zeta(x) \approx |x|$
and to make use of semi-group arguments as well as $L^p$-$L^q$ estimates in order to derive an $L^\infty$ bound for $w$
which in turn implies the desired estimate of the form \eqref{eq:pw_scalar:u_est} for $u$.
However, through their mere nature, these methods are evidently inadequate to handle equations with nonlinear diffusion.

The present paper is built upon the belief that, generally,
an iterative testing procedure should be as strong as semi-group arguments.
While the latter method may be quite elegant,
the former has the distinct advantage of being applicable not only to equations with linear diffusion
but also to \eqref{prob:scalar_eq}.

Indeed, iteratively testing with $w^{p_j-1}$ for certain $1 \le p_j \nea \infty$
allows us to obtain an $L^\infty$ bound for $w$ at the end of Section~\ref{sec:pw_scalar}---%
provided the critical assumption \eqref{eq:pw_scalar:cond_alpha} is fulfilled.

Applying Theorem~\ref{th:pw_scalar} to solutions of \eqref{prob:ks} mainly consists of adequately estimating $f \defs -\nabla v$.
To that end we may basically rely on the results in \cite{WinklerBlowupProfilesLife}.
It probably should also be noted that this is the only part where we explicitly make use of the radially symmetric setting.

Finally, the existence of blow-up profiles
is shown in Section~\ref{sec:ex_profiles}
by considering global solutions $(\ue, \ve)$, $\eps \in (0, 1)$, to suitably approximative problems
which converge (along a subsequence) on all compact sets in $\ombar \setminus \{0\} \times (0, \infty)$ to $(\wh u, \wh v)$
for certain functions $\wh u, \wh v \colon \ombar \times [0, \infty) \ra [0, \infty)$.
We then prove that these functions coincide which $u$ and $v$ on $\ombar \times [0, \tmax)$
such that we may set $U \defs \wh u(\cdot, \tmax)$ as well as $V \defs \wh v(\cdot, \tmax)$
and make use of regularity of $\wh u$ and $\wh v$.

In order to identify $(\wh u, \wh v)$ with $(u, v)$ we crucially need uniqueness of solutions to \eqref{prob:ks}
which we show in Lemma~\ref{lm:unique}---provided that the first equation is nondegenerate.
As this might potentially be of independent interest,
we choose to prove uniqueness for a class of systems slightly generalizing \eqref{prob:ks}.

\section{Pointwise estimates for subsolutions to parabolic equations in divergence form} \label{sec:pw_scalar}
Unless otherwise stated, we assume throughout this section
that $\Omega \subset \R^n$, $n \ge 2$, is a smooth, bounded domain with $0 \in \Omega$,
set $R \defs \sup_{x \in \Omega} |x|$
and suppose that the parameters (all henceforth fixed) in \eqref{eq:pw_scalar:params} as well as $\alpha$
comply with \eqref{eq:pw_scalar:cond_m_q} and \eqref{eq:pw_scalar:cond_alpha}.
Moreover, we may also assume
\begin{align} \label{eq:cond_m_q_alpha_beta}
  (m-q) \alpha \lt \beta,
\end{align}
since whenever \eqref{eq:pw_scalar:f} is fulfilled for some $\beta \gt 0$,
then also for all $\tilde \beta \gt \beta$ (provided one replaces $K_f$ by $\max\{R, 1\}^{\tilde \beta - \beta} K_f$).

In order to simplify the notation,
we also fix $T \in (0, \infty]$ and
functions in \eqref{eq:pw_scalar:reg_dsf} satisfying \eqref{eq:pw_scalar:p} and \eqref{eq:pw_scalar:D1} -- \eqref{eq:pw_scalar:u0}
as well as a nonnegative classical solution $u \in C^0(\ombar \times [0, T)) \cap C^{2, 1}(\ombar \times (0, T))$
of \eqref{prob:scalar_eq},
but emphasize that all constants below
only depend on the parameters in \eqref{eq:pw_scalar:params} as well as on $\alpha$.

Our goal, which will be achieved in Lemma~\ref{lm:w_bdd_l_infty} below, is to prove an $L^\infty$ bound for the function
\begin{align} \label{eq:def_w}
  w \colon \ombar \times [0, T) \ra \R, \quad
  (x, t) \mapsto |x|^\alpha u(x, t)
\end{align}
which in turn directly implies the desired estimate \eqref{eq:pw_scalar:u_est}.

To this end, we will rely on a testing procedure to obtain $L^p$ bounds for all $p \in (1, \infty)$.
Due to an iteration technique, this will then be improved to an $L^\infty$ bound---%
hence the constants in the following proofs need also to be independent of $p$.

In order to prepare said testing procedure, we will need
\begin{lemma} \label{lm:D_S_est}
  Let $s \in \R$ and $0 \le g \in C^0(\ombar \times (0, T) \times (0, \infty))$ 
  with
  \begin{align} \label{eq:D_S_est:cond}
    \sup_{(x, t) \in \Omega \times (0, T)} g(x, t, \rho) \le K_g \max\{\rho, 1\}^s
  \end{align}
  for all $\rho \ge 0$ and some $K_g \gt 0$.

  For any $\mu \in \R$, $\gamma \in \R$ and $\kappa \gt 0$,
  there exist $p_0 \ge 1$ and $C \gt 0$
  such that for all $p \ge p_0$ we have
  \begin{align} \label{eq:D_S_est:statement}
    \intom \left( g(x, t, u) |x|^\mu w^{p+\gamma} \right)^\kappa
    \le C \left( 1 + \intom \left( |x|^{\mu - \alpha s} w^{p+\gamma+s} \right)^\kappa \right)
  \end{align}
  in $(0, T)$.
\end{lemma}
\begin{proof}
  For any $p \gt p_1 \defs -\gamma + \frac{|\mu|}{\alpha}$,
  all integrals in \eqref{eq:D_S_est:statement} are finite by \eqref{eq:def_w}.

  As in the case of $s \le 0$ the statement follows directly by \eqref{eq:D_S_est:cond} and \eqref{eq:def_w}
  (for $p_0 \defs \max\{1, p_1\}$ and $C \defs K_g$),
  we may assume $s \gt 0$.
  Then \eqref{eq:D_S_est:cond} only implies
  \begin{align*}
        \intom \left( g(x, t, u) |x|^\mu w^{p+\gamma} \right)^\kappa
    \le K_g \int_{\{u \ge 1\}} \left( |x|^{\mu - \alpha s} w^{p+\gamma+s} \right)^\kappa
        + K_g \int_{\{u \lt 1\}} \left( |x|^\mu w^{p+\gamma} \right)^\kappa
  \end{align*}
  for all $p \ge p_1$ in $(0, T)$.

  Since $s \gt 0$, we may therein employ Young's inequality (with exponents $\frac{p+\gamma+s}{p+\gamma}, \frac{p+\gamma+s}{s}$)
  to obtain
  \begin{align*}
        \int_{\{u \lt 1\}} \left( |x|^\mu w^{p+\gamma} \right)^\kappa
    \le \frac{p+\gamma}{p+\gamma+s} \intom \left( |x|^{\mu \cdot \frac{p+\gamma+s}{p+\gamma}} w^{p+\gamma+s} \right)^\kappa
        + \frac{s}{p+\gamma+s} |\Omega|
  \end{align*}
  for all $p \ge p_1$ in $(0, T)$.

  As
  \begin{align*}
    \lim_{p \nea \infty} \mu \cdot \frac{p+\gamma+s}{p+\gamma} = \mu \gt \mu - \alpha s
  \end{align*}
  since $\alpha \gt 0$ and $s \gt 0$,
  we may find $p_2 \gt 1$ such that
  $\mu \cdot \frac{p+\gamma+s}{p+\gamma} \gt \mu - \alpha s$ for all $p \gt p_2$.

  Therefore, for $x \in B_1(0)$ and $p \ge p_2$
  \begin{align*}
        |x|^{(\mu \cdot \frac{p+\gamma+s}{p+\gamma}) \kappa }
    \le |x|^{(\mu - \alpha s) \kappa}
  \end{align*}
  while for $x \in \ombar \setminus B_1(0)$ we have for any $p \gt 1$
  \begin{align*}
        |x|^{(\mu \cdot \frac{p+\gamma+s}{p+\gamma}) \kappa}
    \le \max\left\{1, R^{(\mu \cdot \frac{p+\gamma+s}{p+\gamma}) \kappa }\right\}
    \le c_1
    \le c_1 \max\{1, R^{-(\mu - \alpha s) \kappa}\} |x|^{(\mu - \alpha s) \kappa}
  \end{align*}
  for some $c_1 \gt 0$.

  Since $\frac{p+\gamma}{p+\gamma+s} \le 1$ and $\frac{s}{p+\gamma+s} \le 1$,
  we arrive at \eqref{eq:D_S_est:statement} by setting $p_0 \defs \max\{p_1, p_2\}$
  and $C \gt 0$ appropriately.
\end{proof}

We may now initiate the aforementioned testing procedure and obtain a first estimate for the quantity $\ddt \intom w^p$ in $(0, T)$.

\begin{lemma} \label{lm:ddt_wp}
  There exist $C_1, C_2 \gt 0$ and $p_0 \gt 1$ such that for all $p \ge p_0$
  \begin{align} \label{eq:ddt_wp:statement}
          \frac1{p^2} \ddt \intom w^p
          + C_1 \intom |x|^{-(m-1)\alpha} w^{p+m-3} |\nabla w|^2
    &\le  C_2 \sum_{i=1}^3 \left( \intom \left( |x|^{-\mu_i} w^{p+\gamma_i} \right)^{\kappa_i} \right)^\frac1{\kappa_i}
          + C_2
  \end{align}
  in $(0, T)$,
  where
  \begin{align} 
    \mu_1 &\defs (m-1)\alpha + 2, &
    \mu_2 &\defs (2q - m - 1)\alpha + 2\beta, &
    \mu_3 &\defs (q-1)\alpha + 1 + \beta, \vphantom{\frac{\theta}{\theta}} \label{eq:ddt_wp:def_mu} \\
    \gamma_1 &\defs m - 1, &
    \gamma_2 &\defs 2q - m -1, &
    \gamma_3 &\defs q - 1,  \label{eq:ddt_wp:def_gamma} \vphantom{\frac{\theta}{\theta}} \\
    \kappa_1 &\defs 1, &
    \kappa_2 &\defs \frac{\theta}{\theta-2} \text{ and} &
    \kappa_3 &\defs \frac{\theta}{\theta-1}. \label{eq:ddt_wp:def_kappa}
  \end{align}
\end{lemma}
\begin{proof}
  As
  \begin{align*}
      \nabla u
    = \nabla (|x|^{-\alpha} w)
    = |x|^{-\alpha} \nabla w - \alpha |x|^{-\alpha-1} w \nabla |x|,
  \end{align*}
  in $\ombar \times (0, T)$,
  testing the PDI in \eqref{prob:scalar_eq} with $|x|^\alpha w^{p-1}$
  and integrating by parts gives
  \begin{align*}
          \frac1p \ddt \intom w^p
    &=    \intom w_t w^{p-1} \\
    &=    \intom u_t (|x|^\alpha w^{p-1}) \\
    &\le  - \intom (D(x, t, u) \nabla u + S(x, t, u) f) \cdot \nabla( |x|^\alpha w^{p-1} ) \\
    &\pe  + \int_{\partial \Omega} |x|^\alpha w^{p-1} (D(x, t, u) \nabla u + S(x, t, u) f) \cdot \nu
  \intertext{in $(0, T)$, wherein the boundary term is nonpositive because of the second line in \eqref{prob:scalar_eq}.
  Therefore,}
          \frac1p \ddt \intom w^p
    &\le  - (p-1) \intom D(x, t, u) w^{p-2} |\nabla w|^2 \\
    &\pe  + \alpha (p-1) \intom D(x, t, u) |x|^{-1} w^{p-1} \nabla w \cdot \nabla |x| \\
    &\pe  - \alpha \intom D(x, t, u) |x|^{-1} w^{p-1} \nabla w \cdot \nabla |x| \\
    &\pe  + \alpha^2 \intom D(x, t, u) |x|^{-2} w^p |\nabla |x||^2 \\
    &\pe  - (p-1) \intom S(x, t, u) |x|^\alpha w^{p-2} f \cdot \nabla w \\
    &\pe  - \alpha \intom S(x, t, u) |x|^{\alpha-1} w^{p-1} f \cdot \nabla |x|
  \end{align*}
  in $(0, T)$.

  Therein is by Young's inequality
  \begin{align*}
    &\pe  \alpha (p-2) \intom D(x, t, u) |x|^{-1} w^{p-1} \nabla w \cdot \nabla |x| \\
    &\le  \frac{p-1}{2} \intom D(x, t, u) w^{p-2} |\nabla w|^2
          + \frac{\alpha^2 (p-2)^2}{2(p-1)} \intom D(x, t, u) |x|^{-2} w^p |\nabla |x||^2
  \end{align*}
  in $(0, T)$.

  As $|\nabla |x|| = 1$ for all $x \in \Omega \setminus \{0\}$
  and using \eqref{eq:pw_scalar:D1},
  we may therefore find $c_1, c_2, c_3, c_4 \gt 0$ such that
  for all $p \ge 2$
  \begin{align} \label{eq:ddt_wp:c1_c4}
          \frac1{p^2} \ddt \intom w^p
    &\le  - c_1 \intom |x|^{-(m-1)\alpha} w^{p+m-3} |\nabla w|^2 \notag \\
    &\pe  + c_2 \intom  D(x, t, u) |x|^{-2} w^p \notag \\
    &\pe  + c_3 \intom |S(x, t, u)| |x|^\alpha w^{p-2} |f \cdot \nabla w| \notag \\
    &\pe  + c_4 \intom |S(x, t, u)| |x|^{\alpha-1} w^{p-1} |f|
  \end{align}
  holds in $(0, T)$.

  By Lemma~\ref{lm:D_S_est} (with $s = m-1$, $g = D$, $\mu = -2$, $\gamma = 0$, $\kappa = 1$) and \eqref{eq:pw_scalar:D2}
  there are $c_5 \gt 0$ and $p_1 \ge 1$ such that
  \begin{align} \label{eq:ddt_wp:c5}
        \intom  D(x, t, u) |x|^{-2} w^p
    \le c_5 \intom |x|^{-(m-1)\alpha - 2} w^{p+m-1} + c_5
  \end{align}
  for all $p \ge p_1$ in $(0, T)$.

  Furthermore, by employing Young's inequality we may find $c_6 \gt 0$ such that
  \begin{align} \label{eq:ddt_wp:c6}
    &\pe  \intom |S(x, t, u)| |x|^\alpha w^{p-2} |f \cdot \nabla w| \notag \\
    &\le  \frac{c_1}{2c_3} \intom |x|^{-(m-1)\alpha} w^{p+m-3} |\nabla w|^2
          + c_6 \intom |S(x, t, u)|^2 |x|^{(m+1)\alpha} w^{p-m-1} |f|^2
  \end{align}
  for all $p \ge 1$ in $(0, T)$.
  Therein is by Hölder's inequality
  (with exponents $\frac{\theta}{2}, \frac{\theta}{\theta-2}$; note that $\theta \gt n \ge 2$ by \eqref{eq:pw_scalar:params})
  and \eqref{eq:pw_scalar:f}
  \begin{align} \label{eq:ddt_wp:kf_1}
          \intom |S(x, t, u)|^2 |x|^{(m+1)\alpha} w^{p-m-1} |f|^2
    &\le  K_f^\frac{2}{\theta}
          \left(
            \intom \left( |S(x, t, u)|^2 |x|^{(m+1)\alpha - 2\beta} w^{p-m-1} \right)^\frac{\theta}{\theta-2}
          \right)^\frac{\theta-2}{\theta}
  \end{align}
  for all $p \ge 1$ in $(0, T)$.

  Herein we again make use of Lemma~\ref{lm:D_S_est}
  (with $s = 2q $, $g = S^2$, $\mu = (m+1)\alpha - 2\beta$, $\gamma = -m-1$, $\kappa = \frac{\theta}{\theta-2}$)
  and \eqref{eq:pw_scalar:S} to obtain $p_2 \ge 1$ and $c_7 \gt 0$ such that
  \begin{align} \label{eq:ddt_wp:c7}
          \intom \left( |S(x, t, u)|^2 |x|^{(m+1)\alpha - 2\beta} w^{p-m-1} \right)^\frac{\theta}{\theta-2}
    &\le  c_7 \intom \left( |x|^{-(2q-m-1)\alpha - 2\beta} w^{p+2q-m-1} \right)^\frac{\theta}{\theta-2}
          + c_7
  \end{align}
  holds for all $p \ge p_2$ in $(0, T)$.

  Once more employing Hölder's inequality, \eqref{eq:pw_scalar:f},
  Lemma~\ref{lm:D_S_est} (with $s = q$, $g = |S|$, $\mu = \alpha - 1 - \beta$, $\gamma = - 1$, $\kappa = \frac{\theta}{\theta-1}$)
  and \eqref{eq:pw_scalar:S},
  we see that
  \begin{align} \label{eq:ddt_wp:c8}
          \intom |S(x, t, u)| |x|^{\alpha-1} w^{p-1} |f|
    &\le  K_f^\frac1\theta \left(
            \intom \left( |S(x, t, u)| |x|^{\alpha-1 - \beta} w^{p-1} \right)^\frac{\theta}{\theta-1}
          \right)^\frac{\theta-1}{\theta} \notag \\
    &\le  c_8 \left(
            \intom \left( |x|^{-[(q-1) \alpha + 1 + \beta]} w^{p+q-1} \right)^\frac{\theta}{\theta-1}
          \right)^\frac{\theta-1}{\theta}
          + c_8
  \end{align}
  holds for all $p \ge p_3$ in $(0, T)$ for certain $p_3 \ge 1$ and $c_8 \gt 0$.
  
  Finally, by plugging \eqref{eq:ddt_wp:c5} -- \eqref{eq:ddt_wp:c8} into \eqref{eq:ddt_wp:c1_c4},
  we obtain the desired estimate \eqref{eq:ddt_wp:statement}
  for $p_0 \defs \max\{p_1, p_2, p_3\}$ and certain $C_1, C_2 \gt 0$.
\end{proof}

Before estimating the terms on the right hand side of \eqref{eq:ddt_wp:statement} against the dissipative term therein,
we have a deeper look at the parameters in \eqref{eq:ddt_wp:def_mu} -- \eqref{eq:ddt_wp:def_kappa}.
Precisely due to \eqref{eq:pw_scalar:cond_alpha}, our condition on $\alpha$, they allow for the following
\begin{lemma} \label{lm:params}
  Let $i \in \{1, 2, 3\}$
  as well as $\mu_i$ and $\kappa_i$
  as in \eqref{eq:ddt_wp:def_mu} and \eqref{eq:ddt_wp:def_kappa}, respectively.

  Then
  \begin{align} \label{eq:params:def_lambda}
    \lambda_i \defs \frac{\alpha \pu}{\kappa_i (\mu_i - (m-1) \alpha)_+}
  \end{align}
  fulfills
  \begin{align*} 
    \lambda_i \in (1, \infty)
    \quad \text{as well as} \quad
    \frac{2\kappa_i \lambda_i}{\lambda_i-1} \lt \frac{2n}{n-2}.
  \end{align*}
\end{lemma}
\begin{proof}
  Plugging \eqref{eq:ddt_wp:def_mu} into \eqref{eq:params:def_lambda} yields
  \begin{align*}
    \lambda_1 = \frac{\alpha \pu}{2\kappa_1}, \quad
    \lambda_2 = \frac{\alpha \pu}{\kappa_2(2\beta - 2(m-q) \alpha)_+}
    \quad \text{and} \quad
    \lambda_3 = \frac{\alpha \pu}{\kappa_3(1 + \beta - (m-q) \alpha)_+},
  \end{align*}
  hence $\lambda_i \lt \infty$, $i \in \{1, 2, 3\}$,
  since $(m-q) \alpha \lt \beta$ and $\kappa_i \gt 0$, $i \in \{1, 2, 3\}$
  by \eqref{eq:cond_m_q_alpha_beta} and \eqref{eq:ddt_wp:def_kappa}, respectively.

  As $m-q \le \frac{\pu}{\theta} + \frac{\beta \pu-\pu}{n}$ by \eqref{eq:pw_scalar:cond_m_q},
  we furthermore have
  \begin{align*}
        \alpha
    \gt \frac{\beta}{m-q + \frac{\pu}{n} - \frac{\pu}{\theta}}
    \ge \frac{\beta}{\frac{\beta \pu}{n}}
    =   \frac{n}{\pu}
  \end{align*}
  by \eqref{eq:pw_scalar:cond_alpha}.

  Since $\lambda_1 = \frac{\alpha \pu}{2}$ and $\alpha \gt \frac{n}{\pu}$,
  we immediately obtain $\lambda_1 \gt 1$ and
  \begin{align*}
        \frac{2\kappa_1 \lambda_1}{\lambda_1-1}
    =   \frac{2\alpha \pu}{\alpha \pu - 2}
    \lt \frac{2n}{n-2}.
  \end{align*}

  By \eqref{eq:pw_scalar:cond_alpha},
  we have $\alpha \gt \frac{\beta}{m-q + \frac{\pu}{n} - \frac{\pu}{\theta}}$
  and thus due to \eqref{eq:pw_scalar:cond_m_q} also 
  \begin{align*} 
    (m-q) \alpha \gt \beta - \frac{\alpha \pu}{n} + \frac{\alpha \pu}{\theta}.
  \end{align*}

  Therefore, we may further compute
  \begin{align*}
          \kappa_2 \lambda_2 
    &=    \frac{\alpha \pu}{2\beta - 2(m-q) \alpha}
     \gt  \frac{\alpha \pu}{2(\frac{\alpha \pu}{n}-\frac{\alpha \pu}{\theta})}
     =    \frac{n\theta}{2(\theta-n)},
  \end{align*}
  hence $\lambda_2 \gt \frac{(\theta-2)n}{2(\theta-n)} \ge \frac{2(\theta-2)}{2(\theta-2)} = 1$
  since $n \ge 2$
  and
  (as $(\kappa_2, \infty) \ni \xi \mapsto \frac{2\xi}{\frac\xi{\kappa_2} - 1}$ is strictly decreasing)
  \begin{align*}
          \frac{2 \kappa_2 \lambda_2}{\lambda_2-1} 
    &=    \frac{2 \kappa_2 \lambda_2}{\frac{\kappa_2 \lambda_2}{\kappa_2} - 1}
     \lt  \frac{2 n \theta}{\frac{\theta-2}{\theta} n \theta - 2(\theta - n)}
     =    \frac{2 n \theta}{n(\theta-2) -2\theta + 2n} 
     =    \frac{2 n}{n-2}.
  \end{align*}

  Similarly, we see that
  \begin{align*}
          \kappa_3 \lambda_3
    &=    \frac{\alpha \pu}{1 + \beta - (m-q) \alpha} 
     \gt  \frac{\alpha \pu}{1 + \frac{\alpha \pu}{n} - \frac{\alpha \pu}{\theta}}
     \gt  \frac{\alpha \pu}{\frac{2\alpha \pu}{n} - \frac{\alpha \pu}{\theta}}
     =    \frac{n\theta}{2\theta-n}
  \end{align*}
  since $1 \lt \frac{\alpha \pu}{n}$,
  thus
  $\lambda_3 \gt \frac{(\theta-1)n}{2\theta-n} \ge \frac{2\theta-2}{2\theta-2} = 1$
  and
  \begin{align*}
          \frac{2 \kappa_3 \lambda_3}{\lambda_3-1} 
    &=    \frac{2 \kappa_3 \lambda_3}{\frac{\kappa_3 \lambda_3}{\kappa_3} - 1}
     \lt  \frac{2 n \theta}{\frac{\theta-1}{\theta} n \theta - 2\theta + n}
     =    \frac{2 n \theta}{n(\theta-1) - 2\theta + n} 
     =    \frac{2 n}{n-2}.
  \end{align*}

  This clearly proves the lemma.
\end{proof}

Another important ingredient will be
\begin{lemma} \label{lm:lp_bdd_w}
  Throughout $(0, T)$
  \begin{align*} 
    \intom |x|^{-\alpha \pu} w^{\pu} \le M
  \end{align*}
  holds.
\end{lemma}
\begin{proof}
  This is an immediate consequence of \eqref{eq:def_w} and \eqref{eq:pw_scalar:p}.
\end{proof}

As further preparation, we state a quantitative Ehrling-type lemma.
Since this will be also used in the proof of the quite general Lemma~\ref{lm:unique} below
we neither require $n \ge 2$ nor $0 \in \Omega$.
\begin{lemma} \label{lm:ehrling}
  Let $\Omega \subset \R^n$, $n \in \N$, be a smooth, bounded domain
  and $0 \lt s \lt r \lt \frac{2n}{(n-2)_+}$.

  Then there exist $a \in (0, 1)$ and $C \gt 0$ such that for all $\eps \gt 0$ we have
  \begin{align*}
        \|\varphi\|_{\leb r}
    \le \eps \|\nabla \varphi\|_{\leb 2}
        + C \min\{1, \eps\}^{-\frac{a}{1-a}} \|\varphi\|_{\leb s}
    \quad \text{for all $\varphi \in \sob12$}.
  \end{align*}
  Here and below
  we set $\|\varphi\|_{\leb q} \defs \left( \intom |\varphi|^q \right)^\frac1q$ even for $q \in (0, 1)$.
\end{lemma}
\begin{proof}
  The conditions $s \lt r \lt \frac{2n}{(n-2)_+}$
  imply that
  \begin{align*}
          a
    \defs \frac{\frac1s - \frac1r}{\frac 1s + \frac1n - \frac12}
    =     \frac{\frac{r - s}{rs}}{\frac{2n + 2s - ns}{2ns}}
    =     \frac{2nr - 2ns}{2nr + 2rs - nsr}
    =     \frac{r - s}{r - \frac{n-2}{2n} \cdot r s}
  \end{align*}
  satisfies $a \in (0, 1)$.

  Hence we may invoke the Gagliardo--Nirenberg inequality
  (which holds even for $r, s \in (0, 1)$, see for instance \cite[Lemma~2.3]{LiLankeitBoundednessChemotaxisHaptotaxis2016})
  to obtain $c_1 \gt 0$ with the property that
  \begin{align*}
          \|\varphi\|_{\leb r}
    &\le  c_1 \|\nabla \varphi\|_{\leb 2}^{a} \|\varphi\|_{\leb s}^{1-a}
          + c_1 \|\varphi\|_{\leb s}
    \quad \text{for all $\varphi \in \sob12$}.
  \end{align*}
   
  Therein we have by Young's inequality (with exponents $\frac1a, \frac1{1-a}$)
  for all $\eps \in (0, 1)$ and all $\varphi \in \sob12$
  \begin{align*}
          \|\nabla \varphi\|_{\leb 2}^{a} \|\varphi\|_{\leb s}^{1-a}
    &=    \left(\frac{\eps}{a c_1} \|\nabla \varphi\|_{\leb 2}\right)^{a}
          \cdot \left( \left(\frac{\eps}{a c_1}\right)^{-\frac{a}{1-a}} \|\varphi\|_{\leb s}\right)^{1-a} \\
    &\le  \frac{\eps}{c_1} \|\nabla \varphi\|_{\leb 2}
          + c_2 \eps^{-\frac{a}{1-a}} \|\varphi\|_{\leb s},
  \end{align*}
  where $c_2 \defs (1-a) (a c_1)^{\frac{a}{1-a}}$.

  This already implies the statement for $C \defs c_1 (1 + c_2)$.
\end{proof}

In order to be able to apply Lemma~\ref{lm:ehrling},
we first rewrite the dissipative term in \eqref{eq:ddt_wp:statement}.
\begin{lemma} \label{lm:diss}
  There are $c_1, c_2 \gt 0$ and $p_0 \ge 1$ such that for all $p \ge p_0$ we have
  \begin{align*} 
    \left( \Omega \ni x \mapsto |x|^{-\frac{(m-1)\alpha}{2}} w^{\frac{p+m-1}{2}}(x, t) \right) \in \sob12
  \end{align*}
  for all $t \in (0, \tmax)$
  and
  \begin{align*}
          - p^2 \intom |x|^{-(m-1)\alpha} w^{p+m-3} |\nabla w|^2
    &\le  - c_1 \intom \left| \nabla \xw \right|^2
          + c_2 \left( \intom \left( |x|^{-\mu_1} w^{p+\gamma_1} \right)^{\kappa_1} \right)^\frac1{\kappa_1}
  \end{align*}
  in $(0, T)$,
  where $\mu_1$, $\gamma_1$ and $\kappa_1$
  are as in \eqref{eq:ddt_wp:def_mu}, \eqref{eq:ddt_wp:def_gamma} and \eqref{eq:ddt_wp:def_kappa}, respectively.
\end{lemma}
\begin{proof}
  Note first that for $x \in \Omega$ and $t \in (0, T)$ we have
  \begin{align*}
      |x|^{-\frac{(m-1)\alpha}{2}} w^{\frac{p+m-1}{2}}(x, t)
    = |x|^{\frac{\alpha p}{2}} u^{\frac{p+m-1}{2}}(x, t),
  \end{align*}
  hence
  \begin{align*}
    \left( \Omega \ni x \mapsto |x|^{-\frac{(m-1)\alpha}{2}} w^{\frac{p+m-1}{2}}(x, t) \right) \in \con1 \subset \sob12
  \end{align*}
  for all $p \gt p_1 \defs \max\{\frac{2}{\alpha}, 3-m\}$ and all $t \in (0, \tmax)$.
  
  Thus, for $p \ge p_1$, we may calculate
  \begin{align*}
    &\pe  - \intom |x|^{-(m-1)\alpha} w^{p+m-3} |\nabla w|^2 \\
    &=    - \frac{4}{(p+m-1)^2} \intom \left| \nabla\left( |x|^{-\frac{(m-1)\alpha}{2}} w^{\frac{p+m-1}{2}} \right) \right|^2 \\
    &\pe  + \frac{((m-1)\alpha)^2}{(p+m-1)^2} \intom |x|^{-(m-1)\alpha - 2} w^{p+m-1} |\nabla |x||^2 
  \end{align*}
  in $(0, T)$.

  Because of $|\nabla |x|| \equiv 1$ in $\Omega \setminus \{0\}$ and by the definition of $\mu_1, \gamma_1$ and $\kappa_1$
  we have therein
  \begin{align*}
      \intom |x|^{-(m-1)\alpha - 2} w^{p+m-1} |\nabla |x||^2 
    = \left( \intom \left( |x|^{-\mu_1} w^{p+\gamma_1} \right)^{\kappa_1} \right)^\frac1{\kappa_1}
  \end{align*}
  in $(0, T)$ for all $p \ge 1$.

  Moreover, setting $p_2 \defs 2|m-1|$, we have $\frac94 p^2 \ge (p+m-1)^2 \ge \frac14 p^2$ for all $p \ge p_2$,
  so that the statement follows for
  $c_1 \defs \frac{16}{9}$,
  $c_2 \defs 4((m-1)\alpha)^2$
  and $p_0 \defs \max\{1, p_1, p_2\} + 1$.
\end{proof}

A first application of Lemma~\ref{lm:lp_bdd_w} and Lemma~\ref{lm:ehrling} shows
that the dissipative term $\intom \left|\nabla \xw \right|^2$
can be basically turned into $\intom w^p$.
This is the only place where we (directly) need the second condition in \eqref{eq:pw_scalar:cond_m_q},
namely that $m \gt \frac{n-2 \pu}{n}$.

\begin{lemma} \label{lm:wp_dissipitative}
  For given $\eps \gt 0$ and $s \in (0, 2)$,
  we may find $C \gt 0$ and $p_0 \ge 1$ such that
  \begin{align} \label{eq:wp_dissipitative:statement}
    \intom w^p \le \eps \intom \left| \nabla \xw \right|^2 + C \left( \intom \xw^s \right)^\frac2s + C
  \end{align}
  for all $p \ge p_0$ in $(0, T)$.
\end{lemma}
\begin{proof}
  Fix $\eps \gt 0$ as well as $s \in (0, 2)$ arbitrarily 
  and $p_0$ as given by Lemma~\ref{lm:diss}.
  We divide the proof in two parts.

  \emph{Case 1: $m \ge 1$.}
    Young's inequality and Lemma~\ref{lm:ehrling} (with $r = 2 \lt \frac{2n}{n-2}$) imply
    \begin{align*}
            \intom w^p
      &\le  \intom w^{p+m-1} + |\Omega| \\
      &\le  R^{(m-1)\alpha} \intom \xw ^2 + |\Omega| \\
      &\le  \eps \intom \left| \nabla \xw \right|^2 + c_1 \left( \intom \xw^s \right)^\frac2s + |\Omega|
    \end{align*}
    in $(0, T)$ for some $c_1 \gt 0$
    and thus \eqref{eq:wp_dissipitative:statement} for $C \defs \max\{c_1, |\Omega|\}$.

  \emph{Case 2: $m \lt 1$.}
    Since \eqref{eq:pw_scalar:cond_m_q} and $n \ge 2$ assert
    $m \gt \frac{n-2\pu}{n} \ge 1 - \pu$,
    we have $r \defs \frac{2\pu}{m-1+\pu} \in (2, \frac{2n}{n-2})$
    and $\lambda \defs \frac{\pu}{1-m} \in (1, \infty)$.
    We then obtain
    \begin{align}
            \intom w^p \label{eq:wp_dissipitative:m_lt_1_1}
      &\le  \left( \intom |x|^{-\alpha \pu} w^{\pu}  \right)^\frac1\lambda
            \left( \intom |x|^\frac{\alpha \pu}{\lambda-1} w^\frac{p \lambda - \pu}{\lambda-1} \right)^\frac{\lambda-1}\lambda
       \le  M^\frac1\lambda
            \left( \intom \xw^r \right)^\frac{\lambda-1}\lambda
    \end{align}
    for all $p \ge 1$ in $(0, T)$
    by Hölder's inequality as well as Lemma~\ref{lm:lp_bdd_w} and because of
    \begin{align*} 
        \frac{\alpha \pu}{\lambda-1} \cdot \left(-\frac{2}{(m-1)\alpha r} \right)
      = \frac{(m-1)\alpha \pu}{m-1+\pu} \cdot \frac{m-1+\pu}{(m-1)\alpha \pu}
      = 1
    \end{align*}
    as well as
    \begin{align*}
        \frac{p\lambda - \pu}{\lambda - 1} \cdot \frac{2}{(p+m-1) r}
      = \frac{(m-1) (\frac{p\pu}{m-1} + \pu)}{m-1+\pu} \cdot \frac{m-1+\pu}{(p+m-1)\pu}
      = 1.
    \end{align*}

    Noting that $\frac{r(\lambda-1)}{\lambda} = 2$, we again employ Lemma~\ref{lm:ehrling} to see that
    \begin{align} \label{eq:wp_dissipitative:m_lt_1_2}
      &\pe \left( \intom \xw^r \right)^\frac{(\lambda-1)}{\lambda} \notag \\
      &=  \left( \intom \xw^r \right)^\frac2r \notag \\
      &\le \frac{\eps}{M^\frac1\lambda} \intom \left| \nabla \xw \right|^2 + c_2 \left( \intom \xw^s \right)^\frac2s
    \end{align}
    holds in $(0, T)$ for some $c_2 \gt 0$.
    The desired estimate \eqref{eq:wp_dissipitative:statement} is then a direct consequence
    of \eqref{eq:wp_dissipitative:m_lt_1_1} and \eqref{eq:wp_dissipitative:m_lt_1_2}.
\end{proof}

We are now prepared to prove
\begin{lemma} \label{lm:lp_ls_bound}
  For any $0 \lt s \lt s_0 \defs \min\{\frac{2n}{n-2}, \frac{1}{(m-1)_+}\}$,
  we can find $C \gt 0$, $p_0 \gt 1$ and $\nu \ge 1$
  such that for all $p \ge p_0$
  \begin{align} \label{eq:lp_ls_bound:functional}
        \ddt \intom w^p + \intom w^p
    \le C p^\nu + C p^\nu \left(\intom w^{(p+m-1)s-1} \right)^\frac1s
  \end{align}
  in $(0, T)$.
\end{lemma}
\begin{proof}
  By Lemma~\ref{lm:ddt_wp} and Lemma~\ref{lm:diss} there are $c_1, c_2 \gt 0$ and $p_1 \gt 1$ such that for all $p \ge p_1$
  \begin{align} \label{eq:lp_ls_bound:xw_0}
          \ddt \intom w^p
          + c_1 \intom \left|\nabla \xw \right|^2
    &\le  c_2 p^2 \sum_{i=1}^3 \left( \intom \left( |x|^{-\mu_i} w^{p+\gamma_i} \right)^{\kappa_i} \right)^\frac1{\kappa_i}
          + c_2 p^2
  \end{align}
  holds throughout $(0, T)$,
  where $\mu_i, \gamma_i, \kappa_i$, $i \in \{1, 2, 3\}$,
  are given by \eqref{eq:ddt_wp:def_mu}, \eqref{eq:ddt_wp:def_gamma} and \eqref{eq:ddt_wp:def_kappa}, respectively.

  Our goal is to estimate the terms on the right hand side in \eqref{eq:lp_ls_bound:xw_0} against the dissipative term therein.
  As a starting point, we use Hölder's inequality and Lemma~\ref{lm:lp_bdd_w}
  to compute for $\lambda \gt 1$, $p \ge 1$ and $i \in \{1, 2, 3\}$
  \begin{align} \label{eq:lp_ls_bound:xw_est_1}
          \left( \intom \left( |x|^{-\mu_i} w^{p+\gamma_i} \right)^{\kappa_i} \right)^\frac1{\kappa_i}
    &=    \left(
            \intom |x|^{-\frac{\alpha \pu}{\lambda}} w^\frac{\pu}{\lambda}
              \cdot |x|^{-\mu_i \kappa_i + \frac{\alpha \pu}{\lambda}} w^{(p+\gamma_i)\kappa_i - \frac{\pu}{\lambda}}
          \right)^\frac1{\kappa_i} \notag \\
    &\le  M^\frac1{\kappa_i \lambda}
          \left( \intom
            |x|^\frac{-\mu_i \kappa_i \lambda + \alpha \pu}{\lambda - 1}
            w^\frac{(p+\gamma_i)\kappa_i \lambda - \pu}{\lambda -1}
          \right)^\frac{\lambda-1}{\kappa_i \lambda}
  \end{align}
  in $(0, T)$.

  For $p \in (1, \infty)$ and $i \in \{1, 2, 3\}$ set
  \begin{align*}
    \lambda_i(p) \defs
    \begin{cases}
      \dfrac{\alpha p \pu}{\kappa_i [p(\mu_i - (m-1) \alpha) + (m-1) (\mu_i - \alpha \gamma_i) ]_+}, & p \lt \infty, \\[0.9em]
      \dfrac{\alpha \pu}{\kappa_i (\mu_i - (m-1) \alpha)_+}, & p = \infty,
    \end{cases}
  \end{align*}
  then $\lim_{p \nea \infty} \lambda_i(p) = \lambda_i(\infty)$.
  Lemma~\ref{lm:params} asserts $\lambda_i(\infty) \in (1, \infty)$,
  hence there is $p_2 \ge p_1$ such that also $\lambda_i(p) \in (1, \infty)$ for all $p \ge p_2$.

  Setting furthermore
  \begin{align*}
          b_i(p)
    \defs 2 \cdot \frac{\alpha p - (\mu_i - \alpha \gamma_i)}{\alpha p},
    \quad i \in \{1, 2, 3\},
  \end{align*}
  and choosing $\lambda = \lambda_i(p)$ in \eqref{eq:lp_ls_bound:xw_est_1},
  we obtain
  \begin{align} \label{eq:lp_ls_bound:xw_1}
          \left( \intom \left( |x|^{-\mu_i} w^{p+\gamma_i} \right)^{\kappa_i} \right)^\frac1{\kappa_i}
     \le  \max\{M, 1\}
          \left( \intom \left(
            |x|^\frac{-(m-1)\alpha}{2}
            w^\frac{p+m-1}{2}
          \right)^{\frac{\kappa_i \lambda_i(p)}{\lambda_i(p)-1} b_i(p)}
          \right)^\frac{\lambda_i(p)-1}{\kappa_i \lambda_i(p)}
  \end{align}
  in $(0, T)$ for all $p \ge p_2$ and $i \in \{1, 2, 3\}$
  since
  \begin{align*}
    &\pe  \frac{-\mu_i \kappa_i \lambda_i(p) + \alpha \pu}{\lambda_i(p)-1}
          \cdot \frac{2(\lambda_i(p)-1)}{-(m-1)\alpha \kappa_i \lambda_i(p) b_i(p)} \\
    &=    \frac{-\mu_i + \frac{\alpha \pu}{\kappa_i \lambda_i(p)}}{-(m-1)\alpha}
          \cdot \frac{\alpha p}{\alpha p - (\mu_i - \alpha \gamma_i)} \\
    &=    \frac{-\mu_i \alpha p + \alpha[p(\mu_i - (m-1) \alpha) + (m-1) (\mu_i - \alpha \gamma_i)]}{-(m-1)\alpha (\alpha p - (\mu_i - \alpha \gamma_i))}
    =     1
  \end{align*}
  and
  \begin{align*}
    &\pe  \frac{(p+\gamma_i)\kappa_i \lambda_i(p) - \pu}{\lambda_i(p) - 1} \cdot \frac{2(\lambda_i(p)-1)}{(p+m-1) \kappa_i \lambda_i(p) b_i(p)} \\
    &=    \frac{(p+\gamma_i) - \frac{\pu}{\kappa_i \lambda_i(p)}}{p+m-1}
          \cdot \frac{\alpha p}{\alpha p - (\mu_i - \alpha \gamma_i)} \\
    &=    \frac{(p+\gamma_i)\alpha p - [p(\mu_i - (m-1) \alpha) + (m-1) (\mu_i - \alpha \gamma_i)] }{(p+m-1)(\alpha p - (\mu_i - \alpha \gamma_i))}
    =     1 
  \end{align*}
  for all $p \ge p_2$ and $i \in \{1, 2, 3\}$.

  Lemma~\ref{lm:params} further asserts
  \begin{align*}
        \lim_{p \nea \infty} \frac{2\kappa_i \lambda_i(p)}{\lambda_i(p)-1}
    =   \frac{2\kappa_i \lambda_i(\infty)}{\lambda_i(\infty)-1}
    \lt \frac{2n}{n-2}
  \end{align*}
  for all $i \in \{1, 2, 3\}$.
  As moreover \eqref{eq:ddt_wp:def_mu} and \eqref{eq:ddt_wp:def_gamma} entail
  \begin{align*}
    \mu_i - \alpha \gamma_i = 
    \begin{cases}
      2, & i = 1, \\
      2\beta, & i = 2, \\
      1 + \beta, & i = 3
    \end{cases}
  \end{align*}
  and hence $\beta_i(p) \lt 2$ for all $p \ge 1$ and $i \in \{1, 2, 3\}$,
  we may choose $p_3 \ge p_2$ and $r \in (s, \frac{2n}{n-2})$ such that still
  \begin{align*}
    \frac{\kappa_i \lambda_i(p)}{\lambda_i(p)-1} b_i(p) \le r
  \end{align*}
  for all $i \in \{1, 2, 3\}$ and all $p \ge p_3$.

  By Hölder's inequality and the elementary inequality $\xi^A \le 1 + \xi^B$ for $\xi \ge 0$ and $0 \lt A \lt B$,
  we have
  \begin{align} \label{eq:lp_ls_bound:xw_2}
          \left(
            \intom \xw^{\frac{\kappa_i \lambda_i(p)}{\lambda_i(p) - 1} b_i(p)} 
          \right)^\frac{\lambda_i(p)-1}{\kappa_i \lambda_i(p)} 
    &\le  \max\{|\Omega|, 1\} \left( \intom \xw^r \right)^\frac{b_i(p)}r \notag \\
    &\le  c_3 + c_3 \left( \intom \xw^r \right)^\frac2r
  \end{align}
  in $(0, T)$ for all $p \ge p_3$ and $i \in \{1, 2, 3\}$, where $c_3 \defs \max\{|\Omega|, 1\}$.

  Herein we may now finally apply Lemma~\ref{lm:ehrling} together with Young's inequality to obtain $c_4 \gt 0$ such that
  \begin{align} \label{eq:lp_ls_bound:xw_3}
    &\pe  \left( \intom \xw^r \right)^\frac2r \notag \\
    &\le  \frac{c_1}{6c_2c_3 p^2 \max\{M, 1\}} \intom \left|\nabla \xw \right|^2
          + c_4 p^\frac{2a}{1-a} \left( \intom \xw^s \right)^\frac2s
  \end{align}
  in $(0, T)$
  for all $p \ge p_3$.

  By combining \eqref{eq:lp_ls_bound:xw_0}, \eqref{eq:lp_ls_bound:xw_1} -- \eqref{eq:lp_ls_bound:xw_3}
  and Lemma~\ref{lm:wp_dissipitative} (with $\eps = \frac{c_1}{2}$)
  we may find $c_5 \gt 0$ such that
  \begin{align} \label{eq:lp_ls_bound:xw_4}
          \ddt \intom w^p
    +     \intom w^p
    &\le  c_5 p^\frac{4a}{1-a} 
    +     c_5 p^\frac{4a}{1-a} \left( \intom \xw^s \right)^\frac2s
  \end{align}
  in $(0, T)$ for all $p \ge p_3$.

  The assumption $s \le \frac{1}{(m-1)_+}$
  implies $\alpha - (m-1) \alpha s \ge 0$,
  thus again by Hölder's inequality and Lemma~\ref{lm:lp_bdd_w}
  \begin{align*} 
          \left( \intom \xw^s \right)^2
    &\le  \left( \intom |x|^{-\alpha} w \right)
          \left( \intom |x|^{-(m-1)\alpha s + \alpha} w^{(p+m-1)s - 1} \right) \notag \\
    &\le  M^\frac1\pu |\Omega|^\frac{\pu-1}{\pu} R^{\alpha - (m-1)\alpha s}
          \intom w^{(p+m-1)s - 1} 
  \end{align*}
  in $(0, T)$, which together with \eqref{eq:lp_ls_bound:xw_4} implies \eqref{eq:lp_ls_bound:functional}
  for some $C \gt 0$,
  $p_0 \defs p_3$ and $\nu \defs \frac{4a}{1-a}$.
\end{proof}

A direct consequence thereof is
\begin{lemma} \label{lm:w_bdd_lp}
  For all $p \in (1, \infty)$ we have
  \begin{align} \label{eq:w_bdd_lp_statement}
    \sup_{t \in (0, T)} \intom w^p(\cdot, t) \lt \infty.
  \end{align}
\end{lemma}
\begin{proof}
  Let $p_0 \gt 1$ and $s_0 \gt 0$ be as in Lemma~\ref{lm:lp_ls_bound}.
  By Hölder's inequality we may without loss of generality assume that $p \gt p_0$ with $(p+m-1)s_0 - 1 \gt 1$.
  
  Choosing $s \in (0, s_0)$ such that $(p+m-1) s - 1 = 1$
  and noting that
  \begin{align*}
        \intom w
    \le R^\alpha \intom |x|^{-\alpha} w
    \le R^\alpha \intom |x|^{-\alpha} w
    \le R^\alpha M^\frac1{\pu} |\Omega|^{\frac{\pu-1}{\pu}}
  \end{align*}
  in $(0, T)$ by Hölder's inequality and Lemma~\ref{lm:diss},
  we may apply Lemma~\ref{lm:lp_ls_bound} to obtain
  \begin{align*}
    \ddt \intom w^p \le - \intom w^p + C_p
    \quad \text{in $(0, T)$}
  \end{align*}
  for some $C_p \gt 0$
  and hence $\intom w^p \le \max\{\intom w(\cdot, 0)^p, C_p\}$.
  Since $\intom w(\cdot, 0)^p \le |\Omega| \cdot \|w(\cdot, 0)\|_{\leb \infty}^p \le |\Omega| L^p$
  by \eqref{eq:def_w} and \eqref{eq:pw_scalar:u0},
  we may conclude \eqref{eq:w_bdd_lp_statement}.
\end{proof}

Due to a well-established Moser-type iteration technique
(see \cite{AlikakosBoundsSolutionsReactiondiffusion1979} and \cite{MoserNewProofGiorgi1960} for early examples
or also \cite[Lemma~A.1]{TaoWinklerBoundednessQuasilinearParabolic2012}
for an application relevant to quasilinear Keller--Segel systems)
we can also obtain an $L^\infty$ bound for $w$.

\begin{lemma} \label{lm:w_bdd_l_infty}
  There is $C \gt 0$ such that
  \begin{align} \label{eq:w_bdd_l_infty:statement}
    \|w\|_{L^\infty(\Omega \times (0, T))} \lt C.
  \end{align}
\end{lemma}
\begin{proof}
  Set $s \defs \frac12 \min\{\frac1{(m-1)_+}, 1\} \lt \frac{2n}{n-2}$.
  Then Lemma~\ref{lm:lp_ls_bound} asserts the existence of $\tilde p \gt 1$, $c_1 \gt 0$ and $\nu \gt 1$ such that
  \begin{align} \label{eq:w_bdd_l_infty:functional}
    \ddt \intom w^p + \intom w^p \le c_1 p^\nu + c_1 p^\nu \left( \intom w^{(p+m-1)s-1} \right)^\frac1s
  \end{align}
  in $(0, T)$ for all $p \ge \tilde p$.
  
  Set
  \begin{align} \label{eq:w_bdd_l_infty:def_p0}
    p_0 &\defs \max\{\tilde p, 1 - (m-1) s\}
  \intertext{and} \label{eq:w_bdd_l_infty:def_pj}
    p_j &\defs \frac{p_{j-1} + 1 - (m-1) s}{s}
  \end{align}
  for $j \in \N \setminus \{0\}$.
    
  As $s \le \frac1{(m-1)_+}$ and $s \le \frac12$,
  a straightforward induction gives
  \begin{align} \label{eq:w_bdd_l_infty:pj_lower_bdd}
        p_j
    \ge \frac{p_{j-1}}{s}
    \ge \frac{p_0}{s^j}
    \ge 2^j p_0
    \ge 2^j
    \quad \text{for $j \in \N_0$},
  \end{align}
  in particular the sequence $(p_j)_{j \in \N_0}$ is increasing.
  On the other hand by \eqref{eq:w_bdd_l_infty:def_p0} and another induction,
  \begin{align} \label{eq:w_bdd_l_infty:pj_upper_bdd}
        p_j
    \le \frac{p_{j-1} + p_0}{s}
    \le \frac{2p_{j-1}}{s}
    \le \left( \frac{2}{s} \right)^j p_0
    \quad \text{for $j \in \N$}.
  \end{align}

  Since \eqref{eq:w_bdd_l_infty:def_pj} is equivalent to $p_{j-1} = (p_j+m-1)s - 1$, $j \in \N$,
  an ODE comparison argument and \eqref{eq:w_bdd_l_infty:functional} (with $p = p_j$) yield
  \begin{align*} 
        \intom w^{p_j}(\cdot, t)
    \le \max\left\{
          \intom w(\cdot, 0)^{p_j}, \,
          c_1 p_j^\nu + c_1 p_j^\nu \sup_{\tau \in (0, T)} \left( \intom w^{p_{j-1}}(\cdot, \tau) \right)^\frac1s
        \right\}
  \end{align*}
  for all $t \in (0, T]$ and all $j \in \N$.
  Note that Lemma~\ref{lm:w_bdd_lp} asserts finiteness of the right hand side therein.

  Therefore, $A_j \defs \sup_{t \in (0, T)} \|w(\cdot, t)\|_{\leb{p_j}}$, $j \in \N_0$, fulfills
  \begin{align} \label{eq:w_bdd_l_infty:functional_pj}
        A_j
    \le \max\left\{
          \|w(\cdot, 0)\|_{\leb{p_j}}, \,
          (c_1 p_j^\nu)^\frac1{p_j} \left( 1 + A_{j-1}^\frac{p_{j-1}}{s} \right)^\frac1{p_j}
        \right\}
  \end{align}
  for all $j \in \N$.
  
  As $\lim_{p \nea \infty} \|w(\cdot, 0)\|_{\leb p} = \|w(\cdot, 0)\|_{\leb \infty} \le L$
  by \eqref{eq:pw_scalar:u0} and \eqref{eq:def_w},
  there is $c_2 \ge 1$ with
  \begin{align*}
    \|w_0\|_{\leb{p_j}} \le c_2
    \quad \text{for all $j \in \N$}.
  \end{align*}

  Suppose first that there is a strictly increasing sequence $(j_k)_{k \in \N} \subset \N$ such that
  $A_{j_k} \le c_2$ for all $k \in \N$.
  As then
  \begin{align*}
        \|w(\cdot, t)\|_{\leb \infty}
    =   \lim_{k \ra \infty} \|w(\cdot, t)\|_{\leb{p_{j_k}}}
    \le c_2
  \end{align*}
  for all $t \in (0, T)$ since $\lim_{k \ra \infty} p_{j_k} = \infty$ by \eqref{eq:w_bdd_l_infty:pj_lower_bdd},
  this already implies \eqref{eq:w_bdd_l_infty:statement} for $C \defs c_2$.

  Hence, suppose now that on the contrary there is $j_0 \in \N$ such that
  $A_j \gt c_2$ for all $j \ge j_0$.
  Since then also $A_j \ge 1$ for all $j \ge j_0$
  and because of $\frac{p_j}{s} \gt 1$ for all $j \in \N_0$,
  we conclude from \eqref{eq:w_bdd_l_infty:functional_pj} that
  \begin{align*}
    A_j \le (2 c_1 p_j^\nu)^\frac1{p_j} A_{j-1}^{\frac{p_{j-1}}{p_j s}}
    \quad \text{for all $j \gt j_0$.}
  \end{align*}
  As \eqref{eq:w_bdd_l_infty:pj_lower_bdd} entails $\frac{p_{j-1}}{p_j s} \le 1$ we further obtain
  \begin{align*}
    A_j \le (c_2 p_j^\nu)^\frac1{p_j} A_{j-1}
    \quad \text{for all $j \gt j_0$},
  \end{align*}
  where $c_2 \defs 2 c_1$,
  and hence by induction and \eqref{eq:w_bdd_l_infty:pj_upper_bdd}
  \begin{align*}
        A_j
    \le \left(\prod_{i=j_0+1}^{j} (c_2 p_i^\nu)^\frac1{p_i} \right) A_{j_0}
    \le c_3^{\sum_{i=j_0+1}^j \frac{1}{p_i}} \cdot \left(\tfrac2s\right)^{\sum_{i=j_0+1}^j \frac{i \nu}{p_i}} \cdot A_{j_0}
    \quad \text{for all $j \gt j_0$}
  \end{align*}
  with $c_3 \defs c_2 p_0^\nu$.

  As therein by \eqref{eq:w_bdd_l_infty:pj_lower_bdd}
  \begin{align*}
        \sum_{i=j_0+1}^j \frac{1}{p_i}
    \le \sum_{i=j_0+1}^j \frac{i \nu}{p_i}
    \le \sum_{i=1}^\infty \frac{i \nu}{2^i}
    \sfed c_4
    \lt \infty
    \quad \text{for all $j \ge j_0$},
  \end{align*}
  we conclude
  \begin{align*}
        \sup_{t \in (0, T)} \|w(\cdot, t)\|_{\leb \infty} 
    =   \lim_{j \ra \infty} A_j
    \le \left(\frac{2c_3}{s}\right)^{c_4} A_{j_0}
    \lt \infty,
  \end{align*}
  which in turn directly implies the statement.
\end{proof}

The main result of this section now follows immediately.

\begin{proof}[Proof of Theorem~\ref{th:pw_scalar}]
  Combine Lemma~\ref{lm:w_bdd_l_infty} and \eqref{eq:def_w}.
\end{proof}

\section{Pointwise estimates in quasilinear Keller--Segel systems}
Suppose henceforth that $n \ge 2$, $R \gt 0$ and $\Omega \defs B_R(0)$.

In order to apply Theorem~\ref{th:pw_scalar} to the system \eqref{prob:ks}---and hence prove Theorem~\ref{th:pw_ks}---%
we need some integrability information about $\nabla v$.
This is provided by
\begin{lemma} \label{lm:bdd_v}
  Let $K, L, M \gt 0$, $\tilde \alpha \gt \beta \gt n - 1$ and $\theta \in (1, \infty]$.
  Then there is $C \gt 0$ with the following property:

  Suppose that $T \in (0, \infty]$,
  $g \in C^0(\ombar \times [0, T))$ is radially symmetric and nonnegative with 
  \begin{align*}
    \|g(\cdot, t)\|_{\leb1} \le M \quad \text{for all $t \in (0, T)$},
  \end{align*}
  that $v_0 \in \sob1\infty$ is radially symmetric and nonnegative with
  \begin{align*}
    \|v_0\|_{\sob1\infty} \le L
  \end{align*}
  and that, if $\theta = \infty$, 
  \begin{align*}
    g(x, t) \le K |x|^{-\tilde \alpha} \quad \text{for all $x \in \Omega$ and $t \in (0, T)$}.
  \end{align*}

  Then any classical, radially symmetric solution $v \in C^0(\ombar \times [0, T)) \cap C^{2, 1}(\ombar \times (0, T))$ to
  \begin{align*}
    \begin{cases}
      v_t = \Delta v - v + g(x, t), & \text{in $\Omega \times (0, T)$}, \\
      \partial_\nu v = 0,           & \text{in $\partial \Omega \times (0, T)$}, \\
      v(\cdot, 0) = v_0,            & \text{in $\Omega$}
    \end{cases}
  \end{align*}
  fulfills
  \begin{align*}
    \sup_{t \in (0, T)} \intom |x|^{\theta \beta} |\nabla v(x, t)|^\theta \dx \le C
  \end{align*}
  if $\theta \lt \infty$
  and
  \begin{align*}
    \sup_{t \in (0, T)} |\nabla v(x, t)| \le C |x|^{-\beta} \quad \text{for all $x \in \Omega$}
  \end{align*}
  if $\theta = \infty$.
\end{lemma}
\begin{proof}
  See \cite[Lemma~3.4]{WinklerBlowupProfilesLife}.
\end{proof}

We are now indeed able to employ Theorem~\ref{th:pw_scalar}
in order to obtain pointwise estimates for solutions to systems slightly more general than \eqref{prob:ks}.
(The generality is needed as the following Lemma will
be used not only to prove Theorem~\ref{th:pw_ks} but also in in the proof of Lemma~\ref{lm:ue_ve_c2_bdd} below.)

\begin{lemma} \label{lm:pw_ks_g} 
  Suppose that the parameters in \eqref{eq:pw_ks:params} comply with \eqref{eq:pw_ks:cond_m_q} and set $K_g \gt 0$.
  Then for any $\alpha \gt \ul \alpha$,
  with $\ul \alpha$ as in \eqref{eq:pw_ks:cond_alpha},
  and any $\beta \gt n-1$,
  there exists $C \gt 0$ with the following property: 

  Given functions in \eqref{eq:pw_ks:reg} and $g \in C^0([0, \infty))$
  complying with \eqref{eq:pw_scalar:u0}, \eqref{eq:pw_ks:D1} -- \eqref{eq:pw_ks:v_0}
  and
  \begin{align} \label{eq:pw_ks_g:cond_g}
    g(\rho) \le K_g \rho \quad \text{for $\rho \ge 0$},
  \end{align}
  any nonnegative and radially symmetric classical solution
  $(u, v) \in C^0(\ombar \times [0, T)) \cap C^{2, 1}(\ombar \times (0, T))$
  of
  \begin{align} \label{prob:ks_gu} 
    \begin{cases}
      u_t = \nabla \cdot (D(u, v) \nabla u - S(u, v) \nabla v), & \text{in $\Omega \times (0, T)$}, \\
      v_t = \Delta v - v + g(u),                                & \text{in $\Omega \times (0, T)$}, \\
      \partial_\nu u = \partial_\nu v = 0,                      & \text{on $\partial \Omega \times (0, T)$}, \\
      u(\cdot, 0) = u_0, v(\cdot, 0) = v_0,                     & \text{in $\Omega$}
    \end{cases}
  \end{align}
  fulfills \eqref{eq:pw_scalar:u_est} and $|\nabla v(x, t)| \le  C|x|^{-\beta}$ for $x \in \Omega$ and $t \in (0, T)$.
\end{lemma}
\begin{proof}
  We fix such a solution $(u, v)$ and functions in \eqref{eq:pw_ks:reg} as well as $g \in C^0([0, \infty))$,
  but emphasize that all constants below only depend on the parameters in \eqref{eq:pw_ks:params}
  as well as on $K_g, \alpha$ and $\beta$.

  Set $\pu \defs 1$.
  Noting that
  \begin{align*}
      \lim_{\tilde \beta \sea n-1} \lim_{\theta \nea \infty} \frac{\tilde \beta}{m-q + \frac{\pu}{n} - \frac{\pu}{\theta}}
    = \frac{n (n-1)}{(m-q)n + 1}
    = \ul \alpha,
  \end{align*}
  we can choose $\tilde \beta \in (n-1, \beta)$ small enough
  and $\theta \gt n$ large enough
  such that still
  \begin{align*}
    \alpha \gt \frac{\tilde \beta}{m-q + \frac{\pu}{n} - \frac{\pu}{\theta}}.
  \end{align*}

  Setting
  \begin{align*}
    \tilde D(x, t, \rho) \defs D(\rho, v(x, t)), \quad
    \tilde S(x, t, \rho) \defs D(\rho, v(x, t))
    \quad \text{and} \quad
    f(x, t) \defs -\nabla v(x, t)
  \end{align*}
  for $\rho \ge 0, x \in \ombar$ and $t \in (0, T)$
  we see that \eqref{eq:pw_scalar:reg_dsf} -- \eqref{eq:pw_scalar:S} are satisfied (for $\tilde D, \tilde S$ instead of $D, S$),
  while \eqref{eq:pw_scalar:p} follows by \eqref{eq:pw_ks:v_0} and Remark~\ref{rm:pu_1}.
  Furthermore, the boundary conditions in \eqref{prob:ks_gu} imply
  \begin{align*}
    \left( \tilde D(x, t, u) \nabla u + \tilde S(x, t, u) f \right) \cdot \nu = 0 \le 0
    \quad \text{on $\partial \Omega \times (0, T)$.}
  \end{align*}
  As also 
  \begin{align*}
        K_f
    \defs \sup_{t \in (0, T)} \intom |x|^{\theta \tilde \beta} |f(x, t)|^\theta \dx
    =  \sup_{t \in (0, T)} \intom |x|^{\theta \tilde \beta} |\nabla v(x, t)|^\theta \dx
    \lt \infty
  \end{align*}
  by Lemma~\ref{lm:bdd_v},
  we may indeed invoke Theorem~\ref{th:pw_scalar} to obtain
  $C \gt 0$ such that \eqref{eq:pw_scalar:u_est} holds.
  Once more applying Lemma~\ref{lm:bdd_v}, now with $\theta = \infty$, yields
  \begin{align*}
        |\nabla v(x, t)|
    \le C' |x|^{-\tilde \beta}
    \le C' \max\{R, 1\}^{\beta - \tilde \beta} |x|^{-\beta}
    \quad \text{for $x \in \Omega$ and $t \in (0, T)$}
  \end{align*}
  for some $C' \gt 0$.
\end{proof}

An immediate consequence thereof is Theorem~\ref{th:pw_ks}.
\begin{proof}[Proof of Theorem~\ref{th:pw_ks}]
  Choosing $g = \mathrm{id}$ (and, say, $K_g = 1$) in Lemma~\ref{lm:pw_ks_g},
  we see that \eqref{prob:ks_gu} reduces then to \eqref{prob:ks}.
\end{proof}

\section{Existence of blow-up profiles} \label{sec:ex_profiles}
Throughout this section we suppose $n \ge 2$, $R \gt 0$, $\Omega \defs B_R(0)$,
and that \eqref{eq:pw_scalar:u0} and \eqref{eq:pw_ks:D1} -- \eqref{eq:pw_ks:S}
are fulfilled for certain parameters and functions in \eqref{eq:pw_ks:params} and \eqref{eq:pw_ks:reg}, respectively.
In addition---and in contrast to the preceding sections---%
we will also assume \eqref{eq:profile_ks:non_degen}, that is, that $D \ge \eta$, for some $\eta \gt 0$.

Furthermore, fix $\tmax \lt \infty$ and a solution $(u, v)$ to \eqref{prob:ks} (with $\tmax$ instead of $T$)
with the property $\limsup_{t \nea \tmax} \|u(\cdot, t)\|_{\leb \infty} = \infty$.

We will examine whether and in which form $\lim_{t \nea \tmax} u(\cdot, t)$ and $\lim_{t \nea \tmax} v(\cdot, t)$ exist.
To that end we may moreover assume
\begin{align*}
  u_0, v_0 \in \con2
  \quad \text{as well as} \quad
  u, v \in C^{2, 1}(\ombar \times [0, \tmax))
\end{align*}
since the behavior of $(u, v)$ at $\tmax$
may be directly inferred from that of $(\tilde u, \tilde v)$ at $\frac{\tmax}{2}$, where
\begin{align*}
      (\tilde u, \tilde v)
  \defs \left(u(\cdot, \cdot + \tfrac{\tmax}{2}), v(\cdot, \cdot + \tfrac{\tmax}{2})\right)
  \in \left( C^{2, 1}(\ombar \times [0, \tfrac{\tmax}{2})) \right)^2.
\end{align*}

Furthermore, for $\eps \in (0, 1)$,
we fix henceforth $G_\eps \in C^\infty([0, \infty))$ satisfying
$G_\eps(\xi) = \xi$ for all $\xi \in [0, \frac 1\eps]$
and $0 \le G_\eps(\xi) \le \frac2\eps$ for all $\xi \ge 0$.

The main idea is to construct solutions $(\ue, \ve)$, $\eps \in (0, 1)$ to certain approximative problems
which converge along a subsequence to, say, $(\wh u, \wh v)$.
We will then see that these functions coincide with $u$ and $v$ in $\Omega \setminus \{0\} \times (0, \tmax)$
such that, for instance, $\lim_{t \nea \tmax} u(\cdot, t) = \wh u(\cdot, \tmax)$.

\begin{lemma} \label{lm:ue_ve_local_ex}
  For any $\eps \in (0, 1)$ there exists $\tmaxe$
  and a pair of nonnegative functions $(\ue, \ve)$ solving
  \begin{align} \label{eq:lm:ue_ve_local_ex:prob}
    \begin{cases}
      \uet = \nabla \cdot (D(\ue, \ve) \nabla \ue - S(\Ge(\ue), \ve) \nabla \ve), & \text{in $\Omega \times (0, \tmaxe)$}, \\
      \vet = \Delta \ve - \ve + \Ge(\ue),                                         & \text{in $\Omega \times (0, \tmaxe)$}, \\
      \partial_\nu \ue = \partial_\nu \ve = 0,                                    & \text{on $\partial \Omega \times (0, \tmaxe)$}, \\
      \ue(\cdot, 0) = u_0, \ve(\cdot, 0) = v_0,                                   & \text{in $\Omega$}
    \end{cases}
  \end{align}
  classically and having the property
  that if $\tmaxe \lt \infty$ then
  \begin{align*}
    \limsup_{t \nea \tmaxe} \|\ue(\cdot, t)\|_{\leb \infty} = \infty.
  \end{align*}
\end{lemma}
\begin{proof}
  Local existence and extensibility can be proved as in \cite[Lemma~2.1 -- 2.4]{LankeitLocallyBoundedGlobal2017}
  which essentially relies on regularity theory for nondegenerate parabolic equations
  and Schauder's fixed point theorem---while nonnegativity follows by the maximum principle.
\end{proof}

For all $\eps \in (0, 1)$ fix henceforth $\ue, \ve$ and $\tmaxe$ as given by Lemma~\ref{lm:ue_ve_local_ex}.
By quite standard methods we see that the regularized solutions are global in time.
\begin{lemma} \label{lm:ue_ve_global_ex}
  Let $\eps \in (0, 1)$.
  Then the solution $(\ue, \ve)$ constructed in Lemma~\ref{lm:ue_ve_local_ex} is global in time;
  that is, $\tmaxe = \infty$.
\end{lemma}
\begin{proof}
  Since $\Ge$ is bounded,
  $L^p$-$L^q$ estimates (cf.\ \cite[Lemma~1.3~(ii)]{WinklerAggregationVsGlobal2010})
  rapidly yield
  \begin{align*}
    c_1 \defs \sup_{t \in (0, \tmaxe)} \|\ve(\cdot, t)\|_{\sob1\infty} \lt \infty.
  \end{align*}
  Testing the first equation in \eqref{eq:lm:ue_ve_local_ex:prob} with $\ue^{p-1}$, $p \gt 2$, gives
  \begin{align} \label{eq:ue_ve_global_ex:ddt_uep}
          \frac1p \ddt \intom \ue^p
    &=    - (p-1) \intom \ue^{p-2} D(\ue, \ve) |\nabla \ue|^2
          + (p-1) \intom \ue^{p-2} S(G(\ue), \ve) \nabla \ue \cdot \nabla \ve \notag \\
    &\le  - \eta (p-1) \intom \ue^{p-2} |\nabla \ue|^2
          + c_1 c_2 (p-1) \intom \ue^{p-2} |\nabla \ue|
  \end{align}
  in $(0, \tmaxe)$,
  where $c_2 \defs \|S\|_{L^\infty((0, \frac2\eps) \times (0, \infty))}$.

  Therein is by Young's inequality
  \begin{align*}
          \intom \ue^{p-2} |\nabla \ue|
    &\le  \frac{\eta}{4c_1 c_2} \intom \ue^{p-2} |\nabla \ue|^2
          + \frac{c_1c_2}{\eta} \intom \ue^{p-2} \\
    &\le  \frac{\eta}{4c_1 c_2} \intom \ue^{p-2} |\nabla \ue|^2
          + \frac{c_1c_2}{2\eta} \intom \ue^p
          + \frac{c_1c_2|\Omega|}{2\eta}
  \end{align*}
  in $(0, \tmaxe)$,
  so that integrating \eqref{eq:ue_ve_global_ex:ddt_uep} along with an ODE comparison argument yields
  \begin{align} \label{eq:ue_ve_global_ex:u_lp_bdd}
    \sup_{t \in (0, T)} \|u(\cdot, t)\|_{\leb p} \lt \infty \quad \text{for all finite $T \in (0, \tmaxe]$}.
  \end{align}
  By \cite[Lemma~A.1]{TaoWinklerBoundednessQuasilinearParabolic2012},
  this implies \eqref{eq:ue_ve_global_ex:u_lp_bdd} also for $p = \infty$
  so that the extensibility criterion in Lemma~\ref{lm:ue_ve_local_ex} indeed asserts $\tmaxe = \infty$.
\end{proof}

Parabolic regularity allows us to obtain the following
\begin{lemma} \label{lm:ue_ve_c2_bdd}
  For each $\delta \in (0, R)$ and $0 \lt \tau \lt T \lt \infty$ there exist $C \gt 0$ and $\gamma \in (0, 1)$
  such that for all $\eps \in (0, 1)$
  \begin{align} \label{eq:ue_ve_c2_bdd:statement}
    \|\ue\|_{C^{2+\gamma, 1+\frac{\gamma}{2}}(K)} \le C
    \quad \text{and} \quad
    \|\ve\|_{C^{2+\gamma, 1+\frac{\gamma}{2}}(K)} \le C.
  \end{align}
  where $K \defs \ombar \setminus B_\delta(0) \times [\tau, T]$.
\end{lemma}
\begin{proof}
  This can be shown as in \cite[Lemma~4.3]{WinklerBlowupProfilesLife}. We briefly recall the main idea.

  Start by fixing a cutoff function $\zeta \in C^\infty(\ombar \times [0, \infty))$ such that
  \begin{align*}
    \zeta = 1 &\quad \text{in $K$}, \\
    \zeta = 0
     &\quad \text{in $\left(\ol B_{\frac{\delta}{2}}(0) \times [0, \infty)\right)
                      \cup \left(\ombar \times [0, \tfrac{\tau}{2}]\right)$
                  and} \\
    \partial_\nu \zeta = 0 &\quad \text{on $\partial \Omega \times [0, \infty)$}
  \end{align*}
  and set, for $\eps \in (0, 1)$,
  \begin{align*}
    \we \defs \zeta \ue
    \quad \text{as well as} \quad
    \ze \defs \zeta \ve.
  \end{align*}
  
  By Lemma~\ref{lm:pw_ks_g} there exist $c_1, \alpha, \beta \gt 0$ such that
  \begin{align*}
    |\ue(x, t)| \le c_1 |x|^{-\alpha}
    \quad \text{and} \quad
    |\nabla \ve(x, t)| \le c_1 |x|^{-\beta}
  \end{align*}
  for all $x \in \Omega$, $t \in (0, T+1)$ and $\eps \in (0, 1)$.
  
  In particular,
  \begin{align*}
        \sup_{\eps \in (0, 1)} \left( \|\we\|_{L^\infty(\ombar \times [0, T])} + \|\ze\|_{L^\infty(\ombar \times [0, T])} \right)
    \lt \infty.
  \end{align*}

  Basically, the statement follows then by parabolic regularity theory, applied to $\we$ and $\ze$ for $\eps \in (0, 1)$.
  We sketch the main steps.

  At first, \cite[Theorem~1.3]{PorzioVespriHolderEstimatesLocal1993} gives $\tau_1 \in (0, \tau)$ and $\gamma_1 \in (0, 1)$ such that
  \begin{align*}
    \sup_{\eps \in (0, 1)} \|\we\|_{C^{\gamma_1, \frac{\gamma_1}{2}}(\ombar \times [\tau_1, T])} \lt \infty.
  \end{align*}

  In a second step one uses this information along with \cite[Theorem~IV.5.3]{LadyzenskajaEtAlLinearQuasilinearEquations1998}
  to obtain
  \begin{align*}
    \sup_{\eps \in (0, 1)} \|\ze\|_{C^{2+\gamma_2, 1+\frac{\gamma_2}{2}}(\ombar \times [\tau_2, T])} \lt \infty
  \end{align*}
  for some $\tau_2 \in (\tau_1, \tau)$ and $\gamma_2 \in (0, \gamma_1)$.

  Finally, by employing first \cite[Theorem~1.1]{LiebermanHolderContinuityGradient1987}
  and then again \cite[Theorem~IV.5.3]{LadyzenskajaEtAlLinearQuasilinearEquations1998}
  we may find $\tau_2 \lt \tau_3 \lt \tau_4 \lt \tau$ and $0 \lt \gamma_4 \lt \gamma_3 \lt \gamma_2$ such that
  \begin{align*}
    \sup_{\eps \in (0, 1)} \|\we\|_{C^{1+\gamma_3, \frac{1+\gamma_3}{2}}(\ombar \times [\tau_3, T])} \lt \infty
  \end{align*}
  and
  \begin{align*}
    \sup_{\eps \in (0, 1)} \|\we\|_{C^{2+\gamma_4, 1+\frac{\gamma_4}{2}}(\ombar \times [\tau_4, T])} \lt \infty.
  \end{align*}

  Going back to $\ue$ and $\ve$ this indeed gives \eqref{eq:ue_ve_c2_bdd:statement}.
\end{proof}

\begin{lemma} \label{lm:ue_ve_conv}
  There exist $\wh u, \wh v \in C^2(\ombar \setminus \{0\} \times (0, \infty))$ and
  a sequence $(\eps_j)_{j \in \N} \subset (0, 1)$ with $\eps_j \sea 0$ as well as
\begin{align*}
  \uej \ra \wh u
  \quad \text{and} \quad
  \vej \ra \wh v
  \qquad \text{in $C_{\textrm{loc}}^2(\ombar \setminus \{0\} \times (0, \infty))$ as $j \ra \infty$}.
\end{align*}
\end{lemma}
\begin{proof}
  This follows directly from Lemma~\ref{lm:ue_ve_c2_bdd}, the Arzel\`a--Ascoli theorem and a diagonalization argument.
\end{proof}

\begin{lemma} \label{lm:ue_eq_u}
  There exists $\eps_0 \gt 0$ such that
  \begin{align*}
    T_\eps \defs \sup \left\{T \in (0, \tmax) \colon u \le \frac1\eps \text{ in $\ombar \times [0, T]$}\right\}
  \end{align*}
  is well-defined for all $\eps \in (0, \eps_0)$
  and for all $\eps \in (0, \eps_0)$
  \begin{align*}
    \ue = u
    \quad \text{and} \quad
    \ve = v
  \end{align*}
  holds in $\ombar \times [0, T_\eps)$.
\end{lemma}
\begin{proof}
  As $u_0 \equiv 0$ would imply $u \equiv 0$ by Lemma~\ref{lm:unique}
  we may without loss of generality assume $u_0 \not\equiv 0$.
  Then $\eps_0 \defs \frac{1}{2\|u_0\|_{\leb\infty}} \gt 0$ and as $u$ is continuous $T_\eps$ is indeed well-defined for all $\eps \in (0, \eps_0)$.

  Let $\eps \in (0, \eps_0)$.
  In $\ombar \times [0, T_\eps)$ both $(u, v)$ and $(u_\eps, v_\eps)$ are solutions to \eqref{prob:ks} with $T = T_\eps$,
  such that the statement follows due to uniqueness, see Lemma~\ref{lm:unique} below.
\end{proof}

With these preparations at hand, we may now prove Theorem~\ref{th:profile_ks}.
\begin{proof}[Proof of Theorem~\ref{th:profile_ks}]
  Let $\wh u, \wh v$ be given by Lemma~\ref{lm:ue_ve_conv}.
  Since also $\ue \ra u$ and $\ve \ra v$ pointwise (as $\eps \sea 0$) by Lemma~\ref{lm:ue_eq_u},
  we have $u = \wh u$ and $v = \wh v$ in $\ombar \setminus \{0\} \times [0, \tmax)$.

  Because of $\wh u, \wh v \in C^0([0, \tmax]; C^2_{\mathrm{loc}}(\Omega \setminus \{0\}))$
  a consequence thereof is \eqref{eq:profile_ks:u_v_conv}
  if we set $U \defs \wh u(\cdot, \tmax)$ and $V \defs \wh v(\cdot, \tmax)$.
  Finally, \eqref{eq:profile_ks:profile} follows by Theorem~\ref{th:pw_ks}.
\end{proof}

\appendix
\section{Uniqueness in nondegenerate quasilinear Keller--Segel systems} 
As most of the works on quasilinear Keller--Segel systems cited in the introduction do not state whether the solution is unique,
a uniqueness result for quite general systems,
also accounting, for instance, for cell proliferation or consumption of chemicals, 
might be of independent interest.

Since these generalizations do not drastically complicate or enlarge the proof,
we choose to prove a version slightly more general than actually needed for our purposes.

\begin{lemma} \label{lm:unique}
  Suppose $\Omega \subset \R^n$, $n \in \N$, is a smooth, bounded domain.
  Let $\eta \gt 0$, $p \gt \max\{2, n\}$, $T \in (0, \infty]$ as well as
  $D, S, f, g \in C^1([0, \infty)^2)$
  with $D \ge \eta$. 
  Furthermore, assume also that $u_0, v_0 \in \sob1p$ are nonnegative.

  Then there exists at most one pair of nonnegative functions
  \begin{align*}
    (u, v) \in
    \left( C^{2, 1}(\ombar \times (0, T)) \cap C^0([0, T); \sob1p) \right)^2
  \end{align*}
  solving
  \begin{align*}
    \begin{cases}
      u_t = \nabla \cdot (D(u, v) \nabla u - S(u, v) \nabla v) + f(u, v), & \text{in $\Omega \times (0, T)$}, \\
      v_t = \Delta v + g(u, v),                                           & \text{in $\Omega \times (0, T)$}, \\
      \partial_\nu u = \partial_\nu v = 0,                                & \text{on $\partial \Omega \times (0, T)$}, \\
      u(\cdot, 0) = u_0, v(\cdot, 0) = v_0,                               & \text{in $\Omega$}
    \end{cases}
  \end{align*}
  classically.
\end{lemma}
\begin{proof}
  Suppose that $(u_1, v_1)$ and $(u_2, v_2)$ are two such solutions and let $T' \in (0, T)$.

  Due to the supposed regularity and the embedding $\sob1p \embed \con0$ we can find $L \gt 0$ such that
  $u_1, u_2, v_1, v_2 \le L$ in $\ombar \times [0, T']$.

  As then
  \begin{align*}
          (u_1 - u_2)_t
    &=    \nabla \cdot (D(u_1, v_1) \nabla u_1 - S(u_1, v_1) \nabla v_1) + f(u_1, v_1) \\
    &\pe  - \nabla \cdot D(u_2, v_2) \nabla u_2 + S(u_2, v_2) \nabla v_2) - f(u_2, v_2) \\
    &=    \nabla \cdot (D(u_1, v_1) \nabla (u_1 - u_2))
          + \nabla \cdot ((D(u_1, v_1) - D(u_2, v_2)) \nabla u_2) \\
    &\pe  - \nabla \cdot (S(u_1, v_1) \nabla (v_1 - v_2))
          - \nabla \cdot ((S(u_1, v_1) - S(u_2, v_2)) \nabla v_2) \\
    &\pe  + f(u_1, v_1) - f(u_2, v_2)
  \end{align*}
  in $\Omega \times (0, T')$,
  testing with $u_1 - u_2$ and integrating by parts gives
  \begin{align*}
          \frac12 \ddt \intom (u_1 - u_2)^2
    &=    - \intom D(u_1, v_1) |\nabla (u_1 - u_2)|^2 \\
    &\pe  - \intom [D(u_1, v_1) - D(u_2, v_2)] \nabla u_2 \cdot \nabla (u_1 - u_2) \\
    &\pe  + \intom S(u_1, v_1) \nabla (v_1 - v_2) \cdot \nabla (u_1 - u_2) \\
    &\pe  + \intom [S(u_1, v_1) - S(u_2, v_2)] \nabla v_2 \cdot \nabla (u_1 - u_2) \\
    &\pe  + \intom [f(u_1, v_1) - f(u_2, v_2)] (u_1 - u_2) \\
    &\sfed I_1 + I_2 + I_3 + I_4 + I_5
  \end{align*}
  in $(0, T')$.

  Therein we make first use of the nondegeneracy, that is, the crucial assumption that $D \ge \eta$, to see that
  \begin{align*}
    I_1 \le -\eta \intom |\nabla (u_1 - u_2)|^2
  \end{align*}
  holds in $(0, T')$.

  Also, by Young's inequality
  \begin{align*}
        I_3
    \le \frac{\eta}{4} \intom |\nabla (u_1 - u_2)|^2
        + c_1 \intom |\nabla (v_1 - v_2)|^2
  \end{align*}
  holds in $(0, T')$, where $c_1 \defs \frac{\|S\|_{C^0([0, L]^2)}^2}{\eta}$.

  By the mean value theorem we can find $\xi_1, \xi_2 \colon \Omega \times (0, T') \ra [0, L]$ such that
  \begin{align*}
          |D(u_1, v_1) - D(u_2, v_2)|
    &\le  |D(u_1, v_1) - D(u_2, v_1)| + |D(u_2, v_1) - D(u_2, v_2)| \\
    &=    |D_u(\xi_1, v_1) (u_1 - u_2)| + |D_v(u_2, \xi_2) (v_1 - v_2)| \\
    &\le  \|D\|_{C^1([0, L]^2)} \left( |u_1 - u_2| + |v_1 - v_2| \right)
  \end{align*}
  in $\Omega \times (0, T')$,
  where $\|\varphi\|_{C^1([0, L]^2)} \defs
  \max\{\|\varphi\|_{C^0([0, L]^2)}, \|\varphi_u\|_{C^0([0, L]^2)}, \|\varphi_v\|_{C^0([0, L]^2)}\}$
  for $\varphi \in C^1([0, L]^2)$.

  Thus, by Young's and Hölder's inequalities (with exponents $\frac p2, \frac{p}{p-2}$)
  \begin{align*}
          I_2
    &\le  \frac{\eta}{8} \intom |\nabla (u_1 - u_2)|^2
          + c_2 \left( \intom |\nabla u_2|^2 (u_1 - u_2)^2 + \intom |\nabla u_2|^2 (v_1 - v_2)^2 \right) \\
    &\le  \frac{\eta}{8} \intom |\nabla (u_1 - u_2)|^2
          + c_3 \left( \intom (u_1 - u_2)^\frac{2p}{p-2} \right)^\frac{p-2}{p}
          + c_3 \left( \intom (v_1 - v_2)^\frac{2p}{p-2} \right)^\frac{p-2}{p}
  \end{align*}
  in $(0, T')$ with $c_2 \defs \frac{4\|D\|_{C^1([0, L]^2)}^2}{\eta}$
  and $c_3 \defs c_2 \|\nabla u_2\|_{L^\infty((0, T'); \sob1p)}^2$.

  As our assumptions on $p$ imply $r \defs \frac{2p}{p-2} \lt \frac{2n}{(n-2)_+}$,
  we may invoke Lemma~\ref{lm:ehrling} to find $c_4 \gt 0$ with the property that
  \begin{align*}
        \left( \intom |\varphi|^r \right)^\frac2r
    \le \frac{\eta}{8c_3} \intom |\nabla \varphi|^2 
        + c_4 \intom \varphi^2
    \quad \text{for all $\varphi \in \sob12$},
  \end{align*}
  hence
  \begin{align*}
          I_2
    &\le  \frac{\eta}{4} \intom |\nabla (u_1 - u_2)|^2
          + \frac{\eta}{8} \intom |\nabla (v_1 - v_2)|^2
          + c_5 \intom (u_1 - u_2)^2
          + c_5 \intom (v_1 - v_2)^2
  \end{align*}
  in $(0, T')$, where $c_5 \defs c_3 c_4$.

  Similarly, we see that
  \begin{align*}
          I_4
    &\le  \frac{\eta}{4} \intom |\nabla (u_1 - u_2)|^2
          + \frac{\eta}{8} \intom |\nabla (v_1 - v_2)|^2
          + c_6 \intom (u_1 - u_2)^2
          + c_6 \intom (v_1 - v_2)^2
  \end{align*}
  in $(0, T')$ for some $c_6 \gt 0$.

  As again by the mean value theorem
  \begin{align*}
    |f(u_1, v_1) - f(u_2, v_2)| \le \|f\|_{C^1([0, L]^2)} \left( |u_1 - u_2| + |v_1 - v_2| \right)
  \end{align*}
  in $\Omega \times (0, T')$,
  we conclude
  \begin{align*}
          I_5 
    \le   c_7 \intom (u_1 - u_2)^2
          + c_8 \intom (v_1 - v_2)^2
  \end{align*}
  in $(0, T')$,
  where $c_7 \defs \frac32 \|f\|_{C^1([0, L]^2)}$ and $c_8 \defs \frac12 \|f\|_{C^1([0, L]^2)}$.

  Moreover, 
  \begin{align*}
          \frac12 \ddt \intom (v_1 - v_2)^2
    &\le  - \intom |\nabla (v_1 - v_2)|^2
          + \intom (g(u_1, v_1) - g(u_2, v_2)) (v_1 - v_2)
  \end{align*}
  in $(0, T')$.
  
  Therein we make once more use of the mean value theorem to see that
  \begin{align*}
    |g(u_1, v_1) - g(u_2, v_2)| \le \|g\|_{C^1([0, L]^2)} \left( |u_1 - u_2| + |v_1 - v_2| \right)
  \end{align*}
  in $\Omega \times (0, T')$, hence
  \begin{align*}
        \intom (g(u_1, v_1) - g(u_2, v_2)) (v_1 - v_2)
    \le c_9 \intom (u_1 - u_2)^2
        + c_{10} \intom (v_1 - v_2)^2
  \end{align*}
  in $(0, T')$, where $c_9 \defs \frac12 \|g\|_{C^1([0, L]^2)}$ and $c_{10} \defs \frac32 \|g\|_{C^1([0, L]^2)}$.

  By combining the above estimates,
  we obtain with $\lambda \defs c_1 + \frac{\eta}{4}$ and some $c_{11} \gt 0$
  \begin{align*}
        \ddt \left( \intom (u_1 - u_2)^2 + \lambda \intom (v_1 - v_2)^2 \right)
    \le c_{11} \left( \intom (u_1 - u_2)^2 + \lambda \intom (v_1 - v_2)^2 \right)
  \end{align*}
  in $(0, T')$, hence
  \begin{align*}
        \intom (u_1 - u_2)^2(\cdot, t) + \lambda \intom (v_1 - v_2)^2(\cdot, t)
    \le \ure^{c_{11} t} \left( \intom (u_0 - u_0)^2 + \lambda \intom (v_0 - v_0)^2 \right)
    =   0.
  \end{align*}
  for $t \in [0, T']$ by Gr\"onwall's inequality.

  Since $u_1, u_2, v_1, v_2 \in C^0(\ombar \times [0, T'])$,
  this implies $u_1 \equiv u_2$ and $v_1 \equiv v_2$ in $\ombar \times [0, T']$.
  The statement follows upon taking $T' \nea T$.
\end{proof}

\section*{Acknowledgments}
The author is partially supported by the German Academic Scholarship Foundation
and by the Deutsche Forschungsgemeinschaft within the project \emph{Emergence of structures and advantages in
cross-diffusion systems}, project number 411007140.

\footnotesize
\newcommand{\noopsort}[1]{}

\end{document}